\documentclass[11pt]{article}
\usepackage{graphicx}
\usepackage{color}
\usepackage{amsmath, amsthm, amssymb}
\usepackage[round]{natbib}

\newtheorem{theorem}{Theorem}
\newtheorem{remark}{Remark}
\newtheorem{corollary}{Corollary}
\newtheorem{lemma}{Lemma}
\newtheorem{example}{Example}

\numberwithin{equation}{section} \numberwithin{theorem}{section}
\numberwithin{lemma}{section} \numberwithin{remark}{section}
\numberwithin{table}{section} \numberwithin{corollary}{section}
\numberwithin{example}{section} \numberwithin{conjecture}{section}
\numberwithin{assumption}{section}

\newcommand{\ds}{\displaystyle}
\newcommand{\hdg}{h^{\delta}\,\Gamma(3-\delta)}

\begin{document}

\title{A finite difference method for a two-point boundary value problem with a Caputo fractional derivative\thanks{This research was partly supported by the Instituto Universitario de Matem\'aticas y Aplicaciones (IUMA), the  Ministerio de Educaci\'{o}n, Cultura y Deporte (Programa Nacional de Movilidad de Recursos Humanos del Plan Nacional de I+D+i 2008-2011), project MEC/FEDER MTM 2010-16917 and the Diputaci\'{o}n General de Arag\'{o}n.}}

\author{
Martin Stynes\thanks{Corresponding author.  Email: m.stynes@ucc.ie}\\[2pt]
Department of Mathematics, National University of Ireland, Cork, Ireland
\\[6pt]
and\\[6pt]
Jos\'e Luis Gracia\thanks{Email: jlgracia@unizar.es}\\[2pt]
IUMA and Department of Applied Mathematics, University of Zaragoza, Spain
}

\maketitle

\begin{abstract}
{A two-point boundary value problem whose highest-order term is a Caputo fractional derivative of order $\delta \in (1,2)$ is considered. Al-Refai's comparison principle is improved and modified to fit our problem. Sharp a priori bounds on derivatives of the solution $u$ of the boundary value problem are established, showing that $u''(x)$ may be unbounded at the interval endpoint $x=0$. These bounds and a discrete comparison principle are used to prove pointwise convergence of a finite difference method for the problem, where the convective term is discretized using simple upwinding to yield stability on coarse meshes for all values of $\delta$. Numerical results are presented to illustrate the performance of the method.}
Fractional differential equation; Caputo fractional derivative; boundary value problem; derivative bounds; finite difference method; convergence proof.
\end{abstract}

\section{Introduction}\label{sec:intro}
Fractional derivatives are used in an ever-widening range of models of physical processes, and as a consequence the last decade has seen an explosive growth in the number of numerical analysis papers examining differential equations with fractional-order derivatives \citep[see the references in][]{MKM11}. While the analysis of some of these papers \citep[e.g.,][]{MM12,PT12} takes account of the possibly singular behaviour of solutions near some domain boundaries, most fractional-derivative numerical analysis papers work only with very special cases by assuming (explicitly or implicitly) that the solutions they approximate are smooth on the closure of the domain where the problem is posed. In particular, we know of no paper where a finite difference method for a fractional-derivative boundary value problem posed on a bounded domain is analysed rigorously under reasonably general and realistic hypotheses on the behaviour of the solution near the boundaries of that domain. In the present paper we provide the first such rigorous analysis.

Even though we deal with the one-dimensional case---a two-point boundary value problem---the analysis is nevertheless lengthy and requires the development of various techniques that do not appear in the context of ``classical" problems (i.e., problems with integer-order derivatives).

Let $n\in\mathbb{R}$ satisfy $m-1<n<m$ for some positive integer $m$. The Riemann-Liouville fractional derivative~$D^n$ is defined by
$$
D^n g(x) = \left( \frac{d}{dx} \right)^m
    \left[\frac1{\Gamma(m-n)}\int_{t=0}^x (x-t)^{m-n-1}g(t)\,dt \right]
    \quad\text{for }\ 0< x \le 1
$$
for all functions $g$ such that $D^n g(x)$ exists.
Our interest centres on the Caputo fractional derivative~$D_*^n$, which is defined \citep[Definition 3.2]{Diet10} in terms of~$D^n$ by
\begin{equation}\label{CaputoDefn}
D_*^n g = D^n[g - T_{m-1}[g;0]],
\end{equation}
where $T_{m-1}[g;0]$ denotes the Taylor polynomial of degree $m-1$ of the function~$g$ expanded around $x=0$. If $g \in C^{m-1}[0,1]$ and $g^{(m-1)}$ is absolutely continuous on $[0,1]$, then \citep[Theorem~3.1]{Diet10} one also has the equivalent formulation
\begin{equation}\label{CaputoEquiv}
D_*^n g(x) :=  \frac1{\Gamma(m-n)} \int_{t=0}^x (x-t)^{m-n-1} g^{(m)}(t)\, dt
    \quad\text{for }\ 0<x \le 1.
\end{equation}
Our work relies heavily on \citet{PT12}, who use the definition~\eqref{CaputoDefn} of $D_*^n$.

Since the integrals  in~$D^n g(x)$ and $D_\ast^n g(x)$ are associated in a special way with the point $x=0$, many authors write instead~$D_0^n g(x)$ and $D_{\ast\,0}^n g(x)$, but for simplicity of notation we omit the extra subscript~$0$.

Let the parameter $\delta$ satisfy $1 < \delta <2$.
Throughout the paper we consider the two-point boundary value problem
\begin{subequations}\label{prob}
\begin{align}
-D_\ast^\delta &u(x) + b(x)u'(x) + c(x)u(x) = f(x) \ \text{ for }x\in(0,1), \label{proba} \\
&u(0)-\alpha_0u'(0)= \gamma_0,\quad u(1)+\alpha_1u'(1)= \gamma_1, \label{probb}
\end{align}
\end{subequations}
where the constants $\alpha_0, \alpha_1, \gamma_0, \gamma_1$  and the functions $b,c$ and $f$ are given. We assume that $\alpha_1 \ge 0$
and
\begin{equation}\label{AlRcondition}
\alpha_0 \ge \frac1{\delta -1}\,.
\end{equation}
The condition \eqref{AlRcondition} comes from \citet{AlR12a}; it will be used in Sections~\ref{sec:maxprin} and \ref{sec:discrete} below to ensure that \eqref{prob} and its discretization each satisfy a suitable comparison principle. For the moment we assume that $b, c, f \in C[0,1]$; further hypotheses will be placed later on the regularity of these functions.

The problem \eqref{prob} is discussed by \citet{AlR12a} and is a particular case of the wide class of boundary value problems considered in~\citet{PT12}. It is a  steady-state version of the time-dependent problems discussed in~\citet{SD11,SL0405,ZLZ10} and~\citet{JT12}---who describe some advantages of the Caputo fractional derivative over the Riemann-Liouville fractional derivative.

Our paper is structured as follows. Section~\ref{sec:maxprin} obtains a comparison principle for the differential operator and boundary operators in~\eqref{prob}. In Section~\ref{sec:derivatives} existence and uniqueness of a solution to~\eqref{prob} is shown, and sharp pointwise bounds on the integer-order derivatives of this solution are derived. The finite difference discretization of \eqref{prob} on a uniform mesh of width $h$ is described and analysed in Section~\ref{sec:discrete}, and it is proved to be $O(h^{\delta-1})$ convergent at the mesh points. Two numerical examples are presented in Section~\ref{sec:numerical}.

\emph{Notation.} We use the standard notation $C^k(I)$ to denote the space of real-valued functions whose derivatives up to order $k$ are continuous on an interval $I$, and write $C(I)$ for $C^0(I)$.  For each $g\in C[0,1]$, set $\|g\|_\infty = \max_{x\in[0,1]} |g(x)|$. As in \citet{Diet10}, for each positive integer $m$ define
$$
A^m[0,1] = \{ g \in C^{m-1}[0,1]: g^{(m-1)}\text{ is absolutely continuous on } [0,1]  \}.
$$

In several inequalities $C$ denotes a generic constant that  depends on the data of the boundary value problem~\eqref{prob} but is independent of any mesh used to solve~\eqref{prob} numerically; note that $C$ can take different values in different places.

\section{Comparison principle}\label{sec:maxprin}

We begin with a basic result.
\begin{lemma}\label{lem:AlRformula}
\citep[Theorem 2.1]{AlR12b}
Let $g\in C^2[0,1]$ achieve a global minimum at $x_0\in (0,1)$. Then
\begin{equation}\label{AlRformula}
D_\ast^\delta g(x_0) \ge \frac{x_0^{-\delta}}{\Gamma(2-\delta)} \left\{(\delta-1) \left[g(0)-g(x_0)\right] - x_0 g'(0)\right\}.
\end{equation}
\end{lemma}

A careful inspection of the argument used to prove this lemma in \citet{AlR12b} shows that it remains valid under the weaker regularity hypothesis that
$$
\text{(Regularity Hypothesis 1)}\qquad
g\in  C^2(0,1] \text{ and } |g''(x)| \le Cx^{-\theta} \text{ for } 0<x\le 1,
$$
where~$C$ and~$\theta$ are some fixed constants with $0<\theta<1$.  Observe that any function $g$ that satisfies Regularity Hypothesis 1 can be extended to a function (which we also call $g$) lying in $C^1[0,1] \cap A^2[0,1]$. We shall see in Section~\ref{sec:derivatives} that the solution of the boundary value problem~\eqref{prob} satisfies Regularity Hypothesis 1 but does not in general lie in~$C^2[0,1]$.

Lemma~\ref{lem:AlRformula} is the key tool needed to prove the following comparison principle.
\begin{lemma}\label{lem:AlRpositivity1}
\citep[Lemma~3.3]{AlR12a}
Let $g\in C^2[0,1]$. Let $b,c \in C[0,1]$ with $c(x)>0$ for all $x\in(0,1)$. Assume that $g$ satisfies the inequalities
\begin{subequations}\label{AlR}
\begin{align}
&-D_\ast^\delta g+bg' + cg \ge 0\ \text{ on }(0,1), \label{AlR1} \\
g(0)&-\alpha_0 g'(0) \ge 0 \ \text{ and } \ g(1)+\alpha_1 g'(1) \ge 0, \label{AlR2}
\end{align}
\end{subequations}
where $\alpha_0$ satisfies~\eqref{AlRcondition} and $\alpha_1\ge 0$.
Then $g \ge 0$ on $[0,1]$.
\end{lemma}

Recalling our observation above that Lemma~\ref{lem:AlRformula} is still true when the hypothesis $g\in C^2[0,1]$ is replaced by Regularity Hypothesis 1, one sees quickly from the proof of \citet[Lemma~3.3]{AlR12a} that Lemma~\ref{lem:AlRpositivity1} remains valid when the assumption $g\in C^2[0,1]$ is replaced by Regularity Hypothesis 1. In fact, one can go further: Lemma~\ref{lem:AlRformula} shows immediately that when~$g'(0)<0$ one has~$D_\ast^\delta g(x_0) >0$ at the global minimum, and invoking this observation in the proof of \citet[Lemma~3.3]{AlR12a} and changing a few inequalities there from strict to weak or vice versa, the hypothesis $c>0$ can be weakened to $c \ge 0$. That is, one has the following more general version of Lemma~\ref{lem:AlRpositivity1}.

\begin{theorem}\label{th:AlRpositivity}
Let $g$ satisfy Regularity Hypothesis 1. Let $b,c \in C[0,1]$ with $c(x) \ge 0$ for all $x\in(0,1)$. Assume that $g$ satisfies the inequalities~\eqref{AlR}, where $\alpha_0$ satisfies~\eqref{AlRcondition} and $\alpha_1\ge 0$.
Then $g \ge 0$ on $[0,1]$.
\end{theorem}

The next example shows that, for our Caputo differential operator, one does not have a comparison principle for the simplest case of Dirichlet boundary conditions, unlike the situation for classical second-order boundary value problems. Thus one cannot permit $\alpha_0=0$ in Theorem~\ref{th:AlRpositivity}.

\begin{example}\label{exa:maxprincounterexample}
Take $\delta =1.2$. From \citet[Appendix B]{Diet10} we have
$$
D^{1.2}_\ast x^2 = \frac{\Gamma(3)}{\Gamma(1.8)}\, x^{0.8}
\ \text{ and }\
D^{1.2}_\ast x^3 = \frac{\Gamma(4)}{\Gamma(2.8)}\, x^{1.8}
    =\frac{3\Gamma(3)}{(1.8)\Gamma(1.8)}\, x^{1.8}
    \le \frac{(1.67)\Gamma(3)}{\Gamma(1.8)}\, x^{0.8}
$$
for $0 < x < 1$. Also $D^{1.2}_\ast x = 0$. Set $g(x)= x^3-1.67 x^2 + 0.67x$. Then
\begin{equation}\label{Dirichlet}
-D^{1.2}_\ast g(x) \ge 0 \ \text{ for }\ 0 < x <1, \quad g(0) = g(1)=0,
\end{equation}
but $g(0.8) = 0.512 - 1.0688 + 0.536 <0$, so the Dirichlet boundary conditions in~\eqref{Dirichlet} do not justify a comparison principle for $-D^{1.2}_\ast$ on $[0,1]$.

In this example one has $g(0) =0$ and $g'(0)>0$, so the condition $g(0)-\alpha_0 g'(0) \ge 0$ of~\eqref{AlR2} is not satisfied.
\end{example}

\section{A priori bounds on derivatives of the solution}\label{sec:derivatives}
The only source we know for bounds on (certain) derivatives of the solution of \eqref{prob} is~\citet{PT12}, who prove a very general existence result for two-point boundary value problems with differential operators involving fractional-order derivatives. Their analysis is based on \citet{BPV01}.

\emph{Notation.} For integer $q \ge 1$ and $\nu \in (-\infty, 1)$, define $C^{q,\nu}(0,1]$ to be the set of continuous functions $y: [0,1]\to\mathbb{R}$ that are $q$ times continuously differentiable in~$(0,1]$ and satisfy the bounds
\begin{align}
|y^{(i)}(x)| \le
    \begin{cases}
    C &\text{if } i < 1-\nu, \\
    C(1+|\ln x|) &\text{if } i=1-\nu, \\
    Cx^{1-\nu-i} &\text{if } i > 1-\nu, \\
    \end{cases}
\end{align}
for $0<x \le 1$ and $i=1,2,\dots,q$, where $C$ is some constant. Observe that as~$\nu$ increases, the smoothness of functions in~$C^{q,\nu}(0,1]$ decreases. Clearly $C^q[0,1] \subset C^{q,\nu}(0,1] \subset C^{m,\mu}(0,1] \subset C[0,1]$ for $q\ge m \ge 1$ and $\nu \le \mu < 1$.

For our problem \eqref{prob}, the Pedas and Tamme result is as follows.

\begin{theorem}\label{th:PTresult}\cite[Theorem 2.1]{PT12}
Let $b,c, f \in C^{q,\mu}(0,1]$ for some integer $q \ge 1$ and $\mu \in (-\infty, 1)$.  Set $\nu = \max\{\mu, 2-\delta\}$. Let $S$ denote the set of functions $w$ defined on $[0,1]$ for which  $D_\ast^\delta w \in C^{q, \nu}(0,1]$. Assume that the problem~\eqref{prob} with $f\equiv 0,\ \gamma_0=0$ and $\gamma_1=0$ has
in $S$ only the trivial solution $w\equiv 0$.
Then \eqref{prob} has a unique solution $u \in S$; furthermore, $u\in C^1[0,1]$.
\end{theorem}

\begin{remark}\label{rem:onlylinear}
In \citet[Theorem 2.1]{PT12} there is the additional assumption that the only linear polynomial $y(x)$ that satisfies the boundary conditions \eqref{probb} is $y\equiv 0$, but it is straightforward to check that this condition is implied by~$\alpha_1 \ge 0$ and \eqref{AlRcondition}.
\end{remark}

While Theorem~\ref{th:PTresult} bounds $u'$ and the integer-order derivatives of~$D_\ast^\delta u$, it gives no bound on the derivatives~$u^{(i)}$ for $i=2,3,\dots$, but these derivatives will be needed in the consistency analysis of our finite difference method. Thus we now deduce bounds on the integer-order derivatives of $u$ from  Theorem~\ref{th:PTresult}. Our elementary argument can be regarded as interpolating between the integer-order derivatives of~$D_\ast^\delta u$; it relies  only on the derivative bounds stated in Theorem~\ref{th:PTresult} and makes no use of the differential equation~\eqref{proba}.

We shall prove this bound on the integer-order derivatives in a general setting that is suited to fractional-derivative boundary value problems of arbitrary order---such as those considered in \citet{PT12}---since the general proof is essentially the same as the proof for~$D^\delta_\ast$.
At various places in our calculations we shall need the formula
\begin{equation}\label{diffintegral}
\frac{d}{dx}\left[\int_{s=0}^x (x-s)^{\theta_1} r(s)\,ds \right] = \int_{s=0}^x \theta_1(x-s)^{\theta_1-1} r(s)\,ds\quad \text{for }\ 0<x<1,
\end{equation}
when $|r(s)| \le Cs^{-\theta_2}$ for $0<s\le 1$ and the constants $\theta_1, \theta_2$ lie in $(0,1)$; one can justify~\eqref{diffintegral} from the results of~\citet{Ta01} or by writing
$$
\int_{s=0}^x (x-s)^{\theta_1} r(s)\,ds = \int_{s=0}^x  r(s) \int_{t=s}^x \theta_1(t-s)^{\theta_1-1}\,dt\,ds
    =  \int_{t=0}^x \int_{s=0}^t   \theta_1(t-s)^{\theta_1-1} r(s) \,ds\,dt
$$
then applying the fundamental theorem of calculus.

The first step is the following technical result.

\begin{lemma}\label{lem:technical}
Let $m$ be a positive integer and let $\sigma\in\mathbb{R}$ satisfy $m-1 < \sigma < m$.
Suppose that
$$
w(x) = \int_{s=0}^x (x-s)^{\sigma-m} \psi(s)\, ds \quad \text{for }\ 0<x<1,
$$
where $\psi\in C^1(0,1]$ with $|\psi(s)| + s|\psi'(s)| \le C_1s^{\sigma-m}$ for $0<s<1$ and some constant $C_1$. Then
\begin{align}
w'(x) &= \frac1{x} \int_{s=0}^x (x-s)^{\sigma-m} [s\psi'(s) + (\sigma-m+1)\psi(s)]\, ds
\intertext{and}
|w'(x)| &\le C_1 \beta(\sigma-m+1,\sigma-m+1) x^{2(\sigma-m)}\quad \text{for }\ 0<x<1, \label{wprimex}
\end{align}
where $\beta(\cdot, \cdot)$ is Euler's Beta function.
\end{lemma}
\begin{proof}
For $0<x<1$,
\begin{align*}
xw(x) &=  \int_{s=0}^x [(x-s)^{\sigma-m+1}+(x-s)^{\sigma-m}s] \psi(s)\, ds \\
    &= \int_{s=0}^x \left\{(x-s)^{\sigma-m+1}\psi(s) + \frac{(x-s)^{\sigma-m+1}}{\sigma-m+1}
        [s\psi'(s)+\psi(s)]\right\}\, ds,
\end{align*}
after an integration by parts. Applying \eqref{diffintegral} one gets
\begin{align*}
\big(xw(x)\big)'
    &= \int_{s=0}^x (x-s)^{\sigma-m} [s\psi'(s)+(\sigma-m+2)\psi(s)]\, ds,
\intertext{and hence}
w'(x)  &=  \frac1{x} \left[ \big(xw(x)\big)'-w(x) \right]
    = \frac1{x} \int_{s=0}^x (x-s)^{\sigma-m} [s\psi'(s) + (\sigma-m+1)\psi(s)]\, ds,
\end{align*}
as desired. Furthermore, the hypotheses of the lemma imply that
$$
    |w'(x)| \le \frac{C_1}{x}\int_{s=0}^x (x-s)^{\sigma-m} s^{\sigma-m}\, ds
     = C_1 \beta(\sigma-m+1,\sigma-m+1)  x^{2(\sigma-m)},
$$
where the value of the integral is given by Euler's Beta function \citep[Theorem~D.6]{Diet10}.
\end{proof}

The essential property of the bound \eqref{wprimex} is that it takes the form $Cx^{2(\sigma-m)}$ with a constant $C$ that is independent of $x$.

Now we can proceed with the main result of this section.

\begin{theorem}\label{th:purediffbounds}
Let $m$ be a positive integer with $m-1 < \sigma < m$. Assume that $r \in C^{m-1}[0,1]$ and $D_\ast^\sigma r \in C^{q, m-\sigma}(0,1]$ for some integer~$q \ge 1$.
Then $r \in C^{q+m-1}(0,1]$  and for all $x\in (0,1]$ there exists a constant~$C$, which is independent of $x$, such that
\begin{equation}\label{riinterpbound}
|r^{(i)}(x)| \le
    \begin{cases}
    C &\text{if }\ i=0,1,\dots, m-1, \\
    Cx^{\sigma-i} &\text{if }\ i=m, m+1,\dots, q+m-1.
    \end{cases}
\end{equation}
\end{theorem}
\begin{proof}
First, $r\in C^{m-1}[0,1]$ implies~\eqref{riinterpbound} for $i=0,1, \dots, m-1$.

Next, we show that $r \in A^m[0,1]$. Set $w = r-T_{m-1}[r;0]$. Then $w\in C^{m-1}[0,1]$ and $0 = w(0) = w'(0) = \dots = w^{(m-1)}(0)$. By definition
\begin{align*}
D^\sigma w(x) &= \left( \frac{d}{dx} \right)^m
    \left[\frac1{\Gamma(m-\sigma)}\int_{t=0}^x (x-t)^{m-\sigma-1}w(t)\,dt \right] \\
    &= \left( \frac{d}{dx} \right)^m
    \left[\frac1{\Gamma(2m-\sigma-1)}\int_{t=0}^x (x-t)^{2m-\sigma-2}w^{(m-1)}(t)\,dt \right] \\
    &= \frac{d}{dx}
    \left[\frac1{\Gamma(m-\sigma)}\int_{t=0}^x (x-t)^{m-\sigma-1}w^{(m-1)}(t)\,dt \right],
\end{align*}
after $m-1$ integrations by parts followed by $m-1$ differentiations using \eqref{diffintegral}. Consequently
$$
\int_{s=0}^x D^\sigma w(s)\, ds =  \frac1{\Gamma(m-\sigma)}\int_{t=0}^x (x-t)^{m-\sigma-1}w^{(m-1)}(t)\,dt.
$$
This is an Abel integral equation for the function $w^{(m-1)}$. Thus from~\citet[Section 2]{SKR93} it follows that
$$
w^{(m-1)}(x)  = \frac{d}{dx}\left\{\frac1{\Gamma(\sigma+1-m)}\int_{t=0}^x (x-t)^{\sigma-m}
     \left[ \int_{s=0}^t D^\sigma w(s)\, ds \right]  \,dt \right\}.
$$
But $D^\sigma w = D_*^\sigma r \in C[0,1]$ by hypothesis, so we can integrate by parts then use \eqref{diffintegral} to get
\begin{align*}
w^{(m-1)}(x)  &= \frac{d}{dx}\left[\frac1{\Gamma(\sigma+2-m)}\int_{t=0}^x (x-t)^{\sigma+1-m}
        D^\sigma w(t) \,dt \right] \\
    &= \frac1{\Gamma(\sigma+1-m)}\int_{t=0}^x (x-t)^{\sigma-m}D^\sigma w(t) \,dt.
\end{align*}
As the integrand here lies in the space $L_1[0,1]$ of Lebesgue integrable functions, it follows that $w^{(m-1)}$ is absolutely continuous on $[0,1]$. Hence $r^{(m-1)} = (w + T_{m-1}[r;0])^{(m-1)}$ is absolutely continuous on $[0,1]$, i.e., $r\in A^m[0,1]$.

We come now to the main part of the proof.
Set $\phi =D_\ast^\sigma r$. As $r\in A^m[0,1]$, by \citet[Corollary 3.9]{Diet10} one has
$$
r(x) = \sum_{j=0}^{m-1}\frac{r^{(j)}(0)}{j!} x^j + \frac1{\Gamma(\sigma)} \int_{s=0}^x (x-s)^{\sigma-1}\phi(s)\,ds\,.
$$
Integration by parts yields
\begin{align}
r(x) &= \sum_{j=0}^{m-1}\frac{r^{(j)}(0)}{j!} x^j + \frac{\phi(0)}{\sigma\Gamma(\sigma)}\, x^\sigma
        + \frac1{\sigma\Gamma(\sigma)} \int_{s=0}^x (x-s)^{\sigma}\phi'(s)\,ds  \ \text{ for }\ 0 \le x \le 1.
        \notag
\end{align}
Differentiating this formula $m$ times using $\Gamma(n+1) = n\Gamma(n)$ and \eqref{diffintegral}, we obtain
\begin{equation}\label{rm}
r^{(m)}(x) = \frac{\phi(0)}{\Gamma(\sigma-m+1)}\, x^{\sigma-m}
    + \frac1{\Gamma(\sigma-m+1)} \int_{s=0}^x (x-s)^{\sigma-m}\phi'(s)\,ds.
\end{equation}
Hence, since $\phi = D_\ast^\sigma r \in C^{q, m-\sigma}(0,1]$, for some constants $C$ one obtains
\begin{align}
|r^{(m)}(x)| &\le \frac{|\phi(0)|}{\Gamma(\sigma-m+1)}\,x^{\sigma-m}
    + \frac{C}{\Gamma(\sigma-m+1)} \int_{s=0}^x (x-s)^{\sigma-m} s^{\sigma-m}\,ds \notag\\
        &\le Cx^{\sigma-m} + Cx^{2(\sigma-m)+1} \notag\\
        &\le Cx^{\sigma-m}   \label{rxxbound}
\end{align}
as $x^{\sigma-m} > x^{\sigma-m}x^{\sigma+1-m} = x^{2(\sigma-m)+1}$, and \citet[Theorem D.6]{Diet10} was invoked to bound the integral (Euler's Beta function). This is the desired bound~\eqref{riinterpbound} for $i=m$. Furthermore, it is easy to see from~\eqref{rm} that $r^{(m)}\in C(0,1]$.

We now deduce~\eqref{riinterpbound} for $i=m+1,m+2,\dots$ from~\eqref{rm}. Applying Lemma~\ref{lem:technical}  with $\psi(s)= \phi'(s)$ to differentiate~\eqref{rm}, one gets
$$
|r^{(m+1)}(x)| \le C [x^{\sigma-m-1} +  x^{2(\sigma-m)}] \le  Cx^{\sigma-m-1},
$$
which proves~\eqref{riinterpbound} for $i=m+1$, and
\begin{equation}\label{rm1}
r^{(m+1)}(x) = \frac{\phi(0)}{\Gamma(\sigma-m)}\, x^{\sigma-m-1}
    + \frac1{\Gamma(\sigma-m+1)} \cdot \frac1{x} \int_{s=0}^x (x-s)^{\sigma-m}[s\phi''(s)+(\sigma-m+1)\phi'(s)]\,ds,
\end{equation}
from which one can see that $r^{(m+1)}\in C(0,1]$.

Comparing \eqref{rm} and \eqref{rm1}, the relationship between their leading terms is simple, while the integrals in both are $O(x^{2(\sigma-m)+1})$ but the integral in~\eqref{rm1} is multiplied by $1/x$. One now proceeds to differentiate~\eqref{rm1}, invoking Lemma~\ref{lem:technical} with $\psi(s) = s\phi''(s)+(\sigma-m+1)\phi'(s)$; this will yield a rather complicated formula for $r^{(m+2)}(x)$ that involves two integrals, but one sees readily that these integrals are
$$
\frac1{x}\cdot O(x^{2(\sigma-m)}) + \frac1{x^2}\cdot O(x^{2(\sigma-m)+1}),
$$
whence
$$
|r^{(m+2)}(x)| \le C [x^{\sigma-m-2} +  x^{2(\sigma-m)-1}] \le  Cx^{\sigma-m-2},
$$
which proves~\eqref{riinterpbound} for $i=m+2$.
Continuing in this way, each higher derivative of $r$ introduces a further factor $1/x$ in the estimates, and we can derive successively the bounds of~\eqref{riinterpbound} for $i=m+2, m+3, \dots$. The calculation must stop when one reaches an integral involving $\phi^{(q)}(s)$, i.e., when $i=q+m-1$.
\end{proof}

We now apply this result to our boundary value problem~\eqref{prob}.

\begin{corollary}\label{cor:PTintegerbounds}
Let $b,c, f \in C^{q,\mu}(0,1]$ for some integer $q \ge 2$ and $\mu \le 2-\delta$. Assume that $c \ge 0,\ \alpha_1 \ge 0$ and the condition~\eqref{AlRcondition} is satisfied. Then \eqref{prob} has a unique solution $u$ with
$u \in C^1[0,1] \cap C^{q+1}(0,1]$,  and for all $x\in (0,1]$ there exists a constant~$C$ such that
\begin{equation}\label{ubound}
|u^{(i)}(x)| \le
    \begin{cases}
    C &\text{if }\ i=0,1, \\
    Cx^{\delta-i} &\text{if }\ i=2,  3,\dots, q+1.
    \end{cases}
\end{equation}
\end{corollary}
\begin{proof}
Observe that any function in~$C^{q,\mu}(0,1]$ satisfies Regularity Hypothesis~1 of Section~\ref{sec:maxprin} since $\mu \le 2-\delta <1$. Consequently Theorem~\ref{th:AlRpositivity} implies that if $f\equiv 0,\ \gamma_0=0$ and $\gamma_1=0$, then the problem~\eqref{prob} has  in~$C^{q,\mu}(0,1]$ only the trivial solution $u\equiv 0$. Hence Theorem~\ref{th:PTresult} yields existence and uniqueness of a solution $u$ of~\eqref{prob} with $u \in C^1[0,1]$ and $D_\ast^\delta u \in C^{q, 2-\delta}(0,1]$. An appeal to  Theorem~\ref{th:purediffbounds} completes the proof.
\end{proof}

\begin{remark}\label{rem:fastderivbounds}
If we impose the additional hypothesis that $|b(x)| \ge C >0$ on $[0,1]$, it is then possible to give a simple proof of Corollary~\ref{cor:PTintegerbounds} directly from Theorem~\ref{th:PTresult} without using Theorem~\ref{th:purediffbounds}: differentiating~\eqref{proba} then solving for $u''$ yields
\begin{equation}\label{u2x}
u''(x) = \frac1{b(x)}\, \left[ f' +  \left(D_\ast^ \delta u\right)' -c'u - (c+b')u'\right](x),
\end{equation}
and an appeal to the bounds of Theorem~\ref{th:PTresult} yields \eqref{ubound} immediately for the case $i=2$.
One can then differentiate~\eqref{u2x} iteratively and use Theorem~\ref{th:PTresult} to prove~\eqref{ubound} for $i=3,4,\dots$.

If instead $b(x) \equiv 0$ and $|c(x)| \ge C >0$ on $[0,1]$, a similar technique will work---note that $b \equiv 0$ enables the conclusion of Theorem~\ref{th:PTresult} to be strengthened to $D_\ast^\delta u \in C^{q, \nu}(0,1]$ where $\nu = \max\{\mu, 1-\delta\}$ by \citet[Remark 2.2]{PT12}.
\end{remark}

Finally, we give an example to show that the bounds of Theorem~\ref{th:purediffbounds} are sharp.

\begin{example}\label{exa:interpolationsharp}
Let $m$ be a positive integer with $m-1 < \sigma < m$.
Set $r(x) = x^\sigma + x^{2\sigma-m+1}$ for $x \in[0,1]$. Clearly $r \in C^{m-1}[0,1] \cap C^\infty (0,1]$. Then from \cite[Appendix B]{Diet10} one gets
$$
D_\ast^\sigma r(x) = \Gamma(\sigma+1) + \frac{\Gamma(2\sigma-m+2)}{\Gamma(\sigma-m+2)} \, x^{\sigma-m+1}.
$$
Hence $D_\ast^\sigma r \in C[0,1] \cap C^\infty (0,1]$ and $|D_\ast^\sigma r(x)| \le \Gamma(\sigma+1) + [\Gamma(2\sigma-m+2)/\Gamma(\sigma-m+2)]$, while
\begin{align*}
(D_\ast^\sigma r)^{(i)}(x) &= \frac{\Gamma(2\sigma-m+2)}{\Gamma(\sigma-m+2-i)}\,x^{\sigma-m+1-i} \\
\intertext{and}
r^{(i)}(x) &= \frac{\Gamma(\sigma+1)}{\Gamma(\sigma+1-i)}\,x^{\sigma-i} + \frac{\Gamma(2\sigma-m+2)}{\Gamma(2\sigma-m+2-i)}\,x^{2\sigma-m+1-i} \ \text{ for }\  i=1,2,\dots
\end{align*}
Thus the derivatives of $D_\ast^\sigma r$ satisfy the hypotheses of Theorem~\ref{th:purediffbounds} and the derivatives of $r$ agree with~\eqref{riinterpbound} for $i \ge 2$, i.e., the outcome of Theorem~\ref{th:purediffbounds} cannot be sharpened for $i \ge 2$.
\end{example}

\section{Discretization and convergence}\label{sec:discrete}

\subsection{The discretization of the boundary value problem}\label{sec:propertiesA}
Assume the hypotheses of Corollary~\ref{cor:PTintegerbounds}. Let $N$ be a positive integer. Subdivide $[0,1]$ by the uniform mesh $x_j = j/N =: jh$, for $j=0,1,\dots, N$. Then the standard discretization of $-D_\ast^\delta u(x_j)$ for $j=1,2,\dots,N-1$ is \citep[see, e.g.,][]{So12} given by
\begin{equation}\label{DdeltaDisc}
-D_\ast^\delta u(x_j) \approx -\,\frac1{\hdg}
    \sum_{k=0}^{j-1} d_{j-k}\left(u_{k+2}-2u_{k+1}+u_k \right),
\end{equation}
where $u_k$ denotes the computed approximation to $u(x_k)$,  and we set
\begin{equation}\label{djk}
d_r = r_+^{2-\delta} - (r-1)_+^{2-\delta}\ \text{ for all integers } r,
\end{equation}
with
$$
s_+ = \begin{cases}
        s &\text{if } s\ge 0, \\
        0 &\text{if } s <0.
      \end{cases}
$$
Note that $d_r =0$ for $r \le 0$.

Set $g_j = g(x_j)$ for each mesh point $x_j$, where $g$ can be $b,c$ or $f$. To discretize the convective term~$bu'$ we shall use \emph{simple upwinding} \cite[p.47]{RST08}, because the standard approximation~$u'(x_j)\approx (u_{j+1}-u_{j-1})/(2h)$ may yield a non-monotone difference scheme when $\delta$ is near 1; see~\cite{future_paper_GS} for details. Thus we use the approximation
\begin{equation}\label{bu'}
(bu')(x_j) \approx b_jD^\aleph u_j := \begin{cases}
            b_j(u_{j+1}-u_j)/h &\text{if }\, b_j <0,\\
            b_j(u_j-u_{j-1})/h &\text{if }\, b_j \ge 0.
        \end{cases}
\end{equation}
This difference approximation can also be written as
\begin{equation}\label{bu'2}
b_jD^\aleph u_j = - \frac{(b_j+|b_j|)u_{j-1}}{2h} + \frac{|b_j|u_j}{h} + \frac{(b_j-|b_j|)u_{j+1}}{2h}\,.
\end{equation}

The full discretization of~\eqref{proba} is
\begin{equation}\label{fulldisc}
-\,\frac1{\hdg}
    \sum_{k=0}^{j-1} d_{j-k}\left(u_{k+2}-2u_{k+1}+u_k \right) +b_jD^\aleph u_j+c_ju_j= f_j,
\end{equation}
for $j=1,2,\dots,N-1$. The boundary conditions \eqref{probb} are discretized by approximating $u'(0)$ by $(u_1-u_0)/h$ and $u'(1)$ by $(u_N-u_{N-1})/h$.

Let $A = (a_{jk})_{j,k=0}^N$ denote the $(N+1)\times(N+1)$ matrix corresponding to this discretization of \eqref{prob}, i.e., $A\vec u = \vec f$ where
$\vec u := (u_0\ u_1\dots u_N)^T, \ \vec f := ( \gamma_0 \  f_1\ f_2\dots f_{N-1}\ \gamma_1 )^T$ and the superscript $T$ denotes transpose.
Thus the $0^\text{th}$ row of $A$ is $( (1+\alpha_0 h^{-1})\ \ -\alpha_0 h^{-1}\ \  0\ 0\ \dots 0)$ and its
$N^\text{th}$ row is\\ $(0\ 0\ \dots 0 \ \ -\alpha_1 h^{-1}\ \ (1+\alpha_1 h^{-1}))$. For $j=1,2,\dots,N-1$, the entries of the $j^\text{th}$ row of $A$ satisfy
\begin{subequations}\label{Hessenberg}
\begin{align}
 a_{j0} =& \frac{-d_j}{\hdg} -\epsilon_{j1}\frac{b_1+|b_1|}{2h} \,, \label{Hessenberga}\\
 a_{j1} =&\frac{-d_{j-1}+2d_{j}}{\hdg}
     -\epsilon_{j2}\frac{b_2+|b_2|}{2h}+\epsilon_{j1}
        \left(\frac{|b_1|}{h} + c_1\right),  \label{Hessenbergb}\\
 a_{jk} =& \frac{ -d_{j-k}+2d_{j-k+1}-d_{j-k+2}}{\hdg}
    -\epsilon_{j,k+1}\frac{b_j+|b_j|}{2h}
        +\epsilon_{jk} \left(\frac{|b_j|}{h}+c_j \right) +\epsilon_{j,k-1}\frac{b_j-|b_j|}{2h}\notag\\
       &\hspace{85mm} \text{for } k=2,3,\dots,N, \label{Hessenbergc}
\end{align}
\end{subequations}
where we set
$$
\epsilon_{jk}=
\begin{cases}
1 &\text{if } j=k, \\
0 &\text{otherwise}.
\end{cases}
$$
Hence the matrix $A$ is lower Hessenberg.

Observe that~\eqref{fulldisc} implies that
\begin{equation}\label{rowsum}
\sum_{k=0}^N a_{jk} = c_j \ \text{ for } j=1,2,\dots, N-1.
\end{equation}

We shall prove various inequalities for the non-zero entries of $A$. First, $a_{jj} >0$ for all~$j$ by \eqref{Hessenberg}, \eqref{djk} and $c \ge 0$.
From~\eqref{Hessenberga}, one has
\begin{equation}\label{j0}
a_{j0} <0 \quad\text{for }j=1,2,\dots, N-1.
\end{equation}

By \eqref{Hessenbergb},
\begin{equation}\label{a21}
a_{21} = \ds\frac{2^{3-\delta} - 3}{\hdg}- \frac{b_2+|b_2|}{2h} \,,
\end{equation}
so the sign of $a_{21}$ depends on $\delta, h$ and $b$.

\begin{lemma}\label{missing1}
One has $a_{j1} >0$ for $j=3,4,\dots, N-1$.
\end{lemma}
\begin{proof}
For $j=3,4,\dots, N-1$, equations~\eqref{Hessenbergb} and~\eqref{djk} yield
\begin{align}
\hdg a_{j1} &= -(j-1)^{2-\delta}+(j-2)^{2-\delta}+2j^{2-\delta}-2(j-1)^{2-\delta} \notag\\
    &= 2 \left[(j-1)^{2-\delta}-(j-2)^{2-\delta}\right]
        \left[\frac{j^{2-\delta}-(j-1)^{2-\delta}}{(j-1)^{2-\delta}-(j-2)^{2-\delta}} \, - \frac12\right].  \notag
\end{align}
By Cauchy's mean value theorem, for some $\theta\in (j-1, j)$ we have
\begin{align*}
\frac{j^{2-\delta}-(j-1)^{2-\delta}}{(j-1)^{2-\delta}-(j-2)^{2-\delta}}
    &= \frac{(2-\delta)\theta^{1-\delta}}{(2-\delta)(\theta-1)^{1-\delta}}
    = \left( 1 - \frac1{\theta}\right)^{\delta-1}
    > \left( 1 - \frac12\right)^{\delta-1}
    > \frac12\,,
\end{align*}
where we used $\theta > j-1 \ge 2$. It follows that $a_{j1} >0$.
\end{proof}

\begin{lemma}\label{lem:missing2}
One has
$a_{jk} < 0$ for $j=1,2,\dots,N-1$ and $k\in \{2,3,\dots, j-2,j-1, j+1\}$.
\end{lemma}
\begin{proof}
If $k=j+1$, the inequality $a_{j,j+1}<0$ follows easily from \eqref{Hessenbergc} and \eqref{djk}. Thus consider the case $k\in \{2,3,\dots, j-1\}$. Taylor expansions imply that for some $\eta\in (j-k, j-k+2)$ we have
\begin{align*}
-d_{j-k}+2d_{j-k+1}-d_{j-k+2} &= -\frac{d^2}{dr^2}\left[ r^{2-\delta} - (r-1)^{2-\delta} \right]\bigg|_{r=\eta} \\
    &= -(2-\delta)(1-\delta)\left[ \eta^{-\delta} - (\eta-1)^{-\delta} \right] \\
    &<0.
\end{align*}
Hence \eqref{Hessenbergc} implies that $a_{jk} < 0$ for $k\in \{2,3,\dots, j-1\}.$
\end{proof}

This completes the description of the entries $a_{jk}$ defined in~\eqref{Hessenberg}. The sign pattern of $A$ is
\setcounter{MaxMatrixCols}{12}
$$
\begin{pmatrix}
+ & - & 0 & 0 & 0 & 0 & 0 & \hdotsfor{3}  0 & 0\\
- & + & - & 0 & 0 & 0 & 0 & \hdotsfor{3} 0 & 0\\
- &(4.9)& + & - & 0 & 0 & 0 & \hdotsfor{3} 0 & 0\\
- & + & - & + & - & 0 & 0 & \hdotsfor{3} 0 & 0\\
- & + & - & - & + & - & 0 & \hdotsfor{3} 0 & 0\\
\vdots &\vdots&\vdots &\vdots &&&&&&& \vdots\\
- & + & - & - & \hdotsfor{3} & - & + & - & 0\\
- & + & - & - & \hdotsfor{4} & - & + & -\\
0 & 0 & \hdotsfor{7} 0 & - & + \\
\end{pmatrix}.
$$

\subsection{Monotonicity of the discretization matrix $A$}\label{sec:monotoneA}
We shall show that $A^{-1}$ exists and $A^{-1} \ge 0$. Here and subsequently, an inequality between two matrices or vectors means that this inequality holds true for all the corresponding pairs of entries in those matrices or vectors.


 The positive off-diagonal entries in column~1 of $A$ are inconvenient for our analysis. We shall change their signs, while simultaneously simplifying column~0 of $A$, by the following device. Set
$$
A' = E^{(N-1)}E^{(N-2)}\dots E^{(1)}A
$$
where
the elementary matrix $E^{(k)} := (e_{ij}^{(k)})_{i,j=0}^N$ with
$$
e_{ij}^{(k)}  := \delta_{ij} - \frac{a_{k0}}{a_{00}}\, \delta_{ik}\delta_{j0}.
$$
This multiplication of matrix $A$ on the left by elementary matrices  adds a positive multiple of row 0 of $A$ to each lower row to reduce to zero the off-diagonal entries of column~0. Write $A' = (a_{jk}')_{j,k=0}^N$.

Row 0 of $A'$ is $(a_{00}\ a_{01}\ 0\ 0\dots 0)$, where we recall that  $a_{00}= 1+\alpha_0 h^{-1}$ and $a_{01} = -\alpha_0 h^{-1}$.
By construction $a_{j0}'=0$ for $j = 1,2 \dots, N$. For $k >1$ and all $j$ we clearly have $a_{jk}' = a_{jk}$. The remaining entries of column~1 of $A'$ will be examined below.

\begin{lemma}\label{lem:sign1stcol}
The entries of column 1 of the matrix~$A'$ satisfy
$$
a_{11}' > 0 \ \text{ and }\ a_{j1}'  < 0 \ \text{ for } j=2,3,\dots, N-1.
$$
\end{lemma}
\begin{proof}
For $j =1,2,\dots, N-1$, from~\eqref{Hessenberg} one has
\begin{align}
a_{j1}' &= a_{j1} + \frac{\alpha_0 h^{-1}}{1+\alpha_0 h^{-1}}\, a_{j0}  \notag\\
        &= \frac1{\hdg}\left[-d_{j-1} + 2d_j - \frac{\alpha_0 h^{-1}}{1+\alpha_0 h^{-1}} d_j \right]
        -\epsilon_{j2}\frac{b_2+|b_2|}{2h}  \notag\\
        &\hspace{15mm} +\epsilon_{j1}\left(\frac{|b_1|}{h}
             -\frac{\alpha_0 h^{-1}}{1+\alpha_0 h^{-1}}\cdot\frac{b_1+|b_1|}{2h} +c_1\right)
             \notag\\
        &= \frac1{\hdg}\left[d_j \left( 1 + \frac{h}{h+\alpha_0}\right) - d_{j-1}\right] -\epsilon_{j2}\frac{b_2+|b_2|}{2h} \notag\\
        &\hspace{15mm} +\epsilon_{j1}\left(\frac{|b_1|}{h}
            -\frac{\alpha_0 h^{-1}}{1+\alpha_0 h^{-1}}\cdot\frac{b_1+|b_1|}{2h} +c_1\right).
            \label{aj1prime}
\end{align}
Hence $a_{11}' >0$ as $d_1=1,\ d_0=0$ and $c \ge 0$.

We show next that
\begin{equation}\label{aj1prime_term1}
d_j \left( 1 + \frac{h}{h+\alpha_0}\right) - d_{j-1}<0
\quad\text{for }j=2,3,\dots, N-1.
\end{equation}
This inequality is equivalent to
\begin{equation}\label{aji0}
\frac{d_{j-1}}{d_j} > 1 + \frac{h}{h+\alpha_0}\, = 1 + \frac1{N\alpha_0+ 1}\,.
\end{equation}
By Cauchy's mean value theorem, for some $\eta \in (j-2, j-1)$ we have
\begin{equation}
\frac{d_{j-1}}{d_j} = \frac{(j-1)^{2-\delta}-(j-2)^{2-\delta}}{j^{2-\delta}-(j-1)^{2-\delta}}
    = \frac{(2-\delta)\eta^{1-\delta}}{(2-\delta)(\eta+1)^{1-\delta}}
    = \left(1+ \frac1{\eta}\right)^{\delta-1}
    > \left(1 + \frac1{N-2}\right)^{\delta-1} \label{etaineq}
\end{equation}
because $j \le N-1$.
Hence, using the well-known inequality
$$
\frac{t}{1+t} < \ln (1+t) < t \quad\text{for }\ t>0,
$$
one gets
$$
\ln \left( \frac{d_{j-1}}{d_j} \right)  > (\delta-1) \ln \left(1 + \frac1{N-2}\right)
    > (\delta-1) \,\frac{1/(N-2)}{1+ 1/(N-2)}
    = \frac{\delta-1}{N-1}
$$
and
$$
\ln \left( 1 + \frac1{N\alpha_0 +1} \right) < \frac1{N\alpha_0 +1}\,.
$$
Consequently \eqref{aji0} is proved if
$$
\frac{\delta-1}{N-1} \ge \frac1{N\alpha_0 +1}\,,
$$
i.e., if $N-1 \le (\delta-1)(N\alpha_0 +1) = \alpha_0(\delta-1)N+ \delta-1$, which is true by~\eqref{AlRcondition}. Thus~\eqref{aj1prime_term1} is valid.

Combining \eqref{aj1prime} and \eqref{aj1prime_term1}, one obtains immediately $a_{j1}'  < 0$ for $j=2,3,\dots, N-1$.
\end{proof}

Set $\vec 0 := (0\ 0\ \dots 0)^T$. If all the off-diagonal entries of a square matrix $G$ are non-positive and there exists a vector $\vec v \ge \vec 0$ such that $G\vec v > \vec 0$, then \citep[see, e.g.,][]{Fi86} $G$ is an \emph{M-matrix,} i.e.,~$G^{-1}$ exists with $G^{-1} \ge 0$.

\begin{theorem}\label{th:Amonotone}
The matrix $A$ is invertible and $A^{-1} \ge 0$.
\end{theorem}
\begin{proof}
The properties of $A$, the construction of $A'$, and Lemma~\ref{lem:sign1stcol} imply that the entries of the $(N+1)\times (N+1)$ matrix $A'$ are positive on its main diagonal and non-positive off this diagonal.
We see immediately that $\sum_{k=0}^N a_{0k}' = a_{00} + a_{01} =1$;
for $j=1,2,\dots, N-1$ one has
$$
\sum_{k=0}^N a_{jk}' = 0+ \left(a_{j1} + \frac{\alpha_0 h^{-1}}{1+\alpha_0 h^{-1}}\, a_{j0}\right) + \sum_{k=2}^Na_{jk}
    = \left( \frac{\alpha_0 h^{-1}}{1+\alpha_0 h^{-1}} -1 \right)a_{j0} + \sum_{k=0}^Na_{jk},
$$
but $\sum_{k=0}^Na_{jk}=c_j$ by~\eqref{rowsum} and $a_{j0}<0$ by~\eqref{j0}, so
$$
\sum_{k=0}^N a_{jk}' = \left( \frac{\alpha_0 h^{-1}}{1+\alpha_0 h^{-1}} -1 \right)a_{j0}+c_j  >0;
$$
and the final row of $A'$ is $(0\ 0\ \dots 0 \ \ -\alpha_1 h^{-1}\ \ (1+\alpha_1 h^{-1}))$. Hence $A'(1,1,\dots,1)^T > \vec 0$.
Consequently~$A'$ is an M-matrix. Thus~$(A')^{-1}$ exists with $(A')^{-1} \ge 0$.

But
$$
A' = E^{(N-1)}E^{(N-2)}\dots E^{(1)} A,
$$
which implies that
$$
A^{-1} =  (A')^{-1} E^{(N-1)}E^{(N-2)}\dots E^{(1)}
$$
exists, and, since it is a product of matrices with non-negative entries, $A^{-1} \ge 0$.
\end{proof}

The matrix $A$ is said to be \emph{monotone} because $A^{-1} \ge 0$; see, e.g.,~\citet{Fi86}.

\subsection{Error analysis}\label{sec:error}

Define the truncation error $\vec \tau := (\tau_0\ \tau_1\dots\tau_N)^T$ by
\begin{equation}\label{tau}
A(u-\vec u) = \vec\tau,
\end{equation}
where by ``$Au$" we mean that $A$~multiplies the restriction of $u$ to the mesh. Then
\begin{align}
\tau_0 &= (Au)_0  - \gamma_0 =\alpha_0 u'(x_0)  - \frac{\alpha_0}{h} \left[ u(x_1)-u(x_0) \right], \label{tau0} \\
\tau_N &=  (Au)_N-  \gamma_1 =-\alpha_1 u'(x_N)+ \frac{\alpha_1}{h}\left[ u(x_N)-u(x_{N-1}) \right], \label{tauN}
\end{align}
and for $j=1,2,\dots, N-1$,
\begin{equation}\label{tauj}
\tau_j =  (Au)_j - f(x_j) = (Au)_j + D_\ast^\delta u(x_j)-b(x_j)u'(x_j)-c(x_j)u(x_j).
\end{equation}

\begin{lemma}\label{lem:technical1}
There exists a constant $C$, which is independent of $j$,  such that
\begin{equation}\label{sum1}
\sum_{k=1}^{j-1} \big[ (j-k)^{2-\delta} - (j-k-1)^{2-\delta} \big] k^{\delta-3} \le Cj^{1-\delta}
    \ \text{ for all integers } j \ge 2.
\end{equation}
\end{lemma}
\begin{proof}
For $x\in \mathbb{R}$, let $\lceil x\rceil$ denote the smallest integer that is greater than or equal to $x$. By the mean value theorem applied to the function $x \mapsto x^{2-\delta}$ one gets
\begin{align*}
\sum_{k=1}^{\lceil j/2\rceil-1} \big[ (j-k)^{2-\delta} - (j-k-1)^{2-\delta} \big] k^{\delta-3}
    &\le \sum_{k=1}^{\lceil j/2\rceil-1} (2-\delta)(j-k-1)^{1-\delta} k^{\delta-3} \\
    &\le (2-\delta)(j-\lceil j/2\rceil)^{1-\delta} \sum_{k=1}^{\lceil j/2\rceil-1}  k^{\delta-3} \\
    &\le C (2-\delta)(j/2)^{1-\delta} \sum_{k=1}^\infty  k^{\delta-3} \\
    &\le Cj^{1-\delta},
\end{align*}
as $\delta <2$ implies that the infinite series is convergent. For the remaining terms in~\eqref{sum1}, by telescoping and $1<\delta<2$ we have
\begin{align*}
\sum_{k=\lceil j/2\rceil}^{j-1} \big[ (j-k)^{2-\delta} - (j-k-1)^{2-\delta} \big] k^{\delta-3}
    &\le \lceil j/2\rceil^{\delta-3} \sum_{k=\lceil j/2\rceil}^{j-1} \big[ (j-k)^{2-\delta} - (j-k-1)^{2-\delta} \big]  \\
    &= \lceil j/2\rceil^{\delta-3} \big(j-\lceil j/2\rceil\big)^{2-\delta} \\
    &\le (j/2)^{\delta-3}(j/2)^{2-\delta} \\
    &=2j^{-1}  \\
    &\le  2j^{1-\delta}.
\end{align*}

Adding these two bounds yields \eqref{sum1}.
\end{proof}

We can now bound the truncation errors of the finite difference scheme.
\begin{lemma}\label{lem:TruncError}
There exists a constant $C$ such that
the truncation errors in the discretization $A\vec u = \vec f$ of~\eqref{prob} satisfy
\begin{equation}\label{truncation}
|\tau_j| \le
\left\{ \begin{array}{ll}
C h^{\delta-1} & \hbox{if } j=0, \\
C & \hbox{if } j=1, \\
C (j-1)^{1-\delta} & \hbox{if } j=2,\ldots,N-1,\\
C \alpha_1 h & \hbox{if } j=N.
\end{array} \right.
\end{equation}
\end{lemma}
\begin{proof}
We shall consider separately the four cases  in \eqref{truncation}.

\emph{Case $j=0$:} By the mean value theorem, for some $\eta_1\in (0, h)$ we have
\begin{align*}
|\tau_0| &= \alpha_0\left| u'(0) -  \frac{u(h)-u(0)}{h} \right|
    = \alpha_0 | u'(0) -  u'(\eta_1)|
    = \alpha_0\left|\int_{t=0}^{\eta_1}u''(t)\,dt\right| \\
    &\le \alpha_0\int_{t=0}^h Ct^{\delta-2}\,dt
    \le C\alpha_0 h^{\delta-1},
\end{align*}
where we used Corollary~\ref{cor:PTintegerbounds}.

\emph{Case $j=N$:}  Similarly to the case $j=0$, one gets
$$
|\tau_N| \le  \alpha_1\int_{t=1-h}^1 Ct^{\delta-2}\,dt \le C\alpha_1 h.
$$

\emph{Case $1<j<N$:}  Fix $j\in \{2,3,\dots, N-1\}$. By virtue of~\eqref{tauj} and~\eqref{DdeltaDisc} we can write $\tau_j = R_j + \sum_{k=0}^{j-1}\tau_{j,k}$, where
$$
R_j= b_jD^\aleph u_j - (bu')(x_j)
$$
and
\begin{equation}\label{taujk2}
\tau_{j,k} = \frac1{\Gamma(2-\delta)}\int_{x_k}^{x_{k+1}} (x_j-s)^{1-\delta} u''(s)\,ds  - \frac{d_{j-k}}{\hdg} \left[ u(x_{k+2}) - 2u(x_{k+1}) + u(x_{k})\right].
\end{equation}
By Taylor expansions, for some constant $C$ one has
\begin{align*}
\vert R_j\vert & \le Ch \Vert b \Vert_\infty \max_{x\in[x_{j-1}, x_{j+1}]} |u''(x)|.
\end{align*}
Hence Corollary~\ref{cor:PTintegerbounds}, $hj \le 1$ and $\delta>1$ imply that
\begin{equation}\label{Rj}
\vert R_j\vert \le C h [h (j-1)]^{\delta-2} = C[h(j-1)]^{\delta-1}(j-1)^{-1}
    \le C(j-1)^{1-\delta} .
\end{equation}
Returning to the other component of $\tau_j$, we have
\begin{align*}
\tau_{j,k} &= \frac{u''(\eta_2)}{\Gamma(2-\delta)}\int_{x_k}^{x_{k+1}} (x_j-s)^{1-\delta} \,ds  - \frac{d_{j-k}}{\hdg} \left[ u(x_{k+2}) - 2u(x_{k+1}) + u(x_{k})\right]  \\
      &= \frac{h^{2-\delta}d_{j-k}}{\Gamma(3-\delta)}\left[ u''(\eta_2)  - \frac{u(x_{k+2}) - 2u(x_{k+1}) + u(x_k)}{h^2}  \right]
\end{align*}
for some $\eta_2\in (x_k, x_{k+1})$, by the mean value theorem for integrals.
Taylor expansions show that for some  $\eta_3\in (x_k, x_{k+2})$  and $k \ge 1$ one obtains
$$
\left| u''(\eta_2) - \frac{u(x_{k+2}) - 2u(x_{k+1}) + u(x_k)}{h^2}  \right|
    =  \left| u''(\eta_2) -  u''(\eta_3)\right|
    = \left| (\eta_2 -\eta_3)u'''(\eta_4)\right|
    \le Ch x_k^{\delta-3},
$$
where we again used the mean value theorem to get $\eta_4\in (x_k, x_{k+2})$, then the bound on~$u'''$ given by Corollary~\ref{cor:PTintegerbounds}. Consequently
\begin{equation}\label{taujk}
\vert \tau_{j,k} \vert  \le \frac{Ch^{3-\delta}d_{j-k} x_k^{\delta-3} }{\Gamma(3-\delta)} =  \frac{Cd_{j-k} k^{\delta-3} }{\Gamma(3-\delta)}\quad \text{for } k \ge 1.
\end{equation}
Now an appeal to Lemma~\ref{lem:technical1} gives
\begin{equation}\label{sumtaujk}
\sum_{k=1}^{j-1} \vert \tau_{j,k} \vert \le C j^{1-\delta}.
\end{equation}

To complete the case  $1<j<N$, it remains to bound $\tau_{j,0}$. By~\eqref{taujk2} and a triangle inequality one has
\begin{equation}\label{tauj0}
\vert \tau_{j,0} \vert \le \left|\frac1{\Gamma(2-\delta)}\int_{x_0}^{x_{1}} (x_j-s)^{1-\delta} u''(s)\,ds\right| + \left| \frac{d_{j}}{\hdg} \left[ u(x_{2}) - 2u(x_{1}) + u(x_{0})\right] \right|.
\end{equation}
We bound these two terms separately. First, by Corollary~\ref{cor:PTintegerbounds} and $j>1$ we get
\begin{align*}
\left|\frac1{\Gamma(2-\delta)}\int_{x_0}^{x_{1}} (x_j-s)^{1-\delta} u''(s)\,ds\right|
    &\le \frac{C[(j-1)h]^{1-\delta}}{\Gamma(2-\delta)}\int_{x_0}^{x_{1}}s^{\delta-2}\,ds
    = \frac{C(j-1)^{1-\delta}}{(\delta-1)\Gamma(2-\delta)}\,.
\end{align*}
For the second term in~\eqref{tauj0}, the mean value theorem and Corollary~\ref{cor:PTintegerbounds} give
\begin{align*}
\left| \frac{d_{j}}{\hdg} \left[ u(x_{2}) - 2u(x_{1}) + u(x_{0})\right] \right|
    &= \frac{d_{j}h\vert u'(\eta_6)-u'(\eta_5) \vert}{\hdg} \\
    &\le \frac{d_{j}h}{\hdg}\int^{\eta_6}_{t=\eta_5} |u''(t)|\, dt \\
    &\le \frac{Cd_{j}h}{\hdg}\int^{2h}_{t=0} t^{\delta-2}\, dt \\
    &= C d_{j}  \\
    &\le C (j-1)^{1-\delta},
\end{align*}
where $\eta_5\in (x_0,x_1)$ and $\eta_6\in (x_1,x_2)$.
Combining these inequalities with~\eqref{tauj0} yields
$$
\vert \tau_{j,0} \vert \le C (j-1)^{1-\delta}.
$$
Add this bound to~\eqref{Rj} and~\eqref{sumtaujk} to obtain finally
\begin{equation}\label{tauj3}
\vert \tau_j\vert \le \sum_{k=0}^{j-1} |\tau_{j,k}| + |R_j| \le C (j-1)^{1-\delta}
\quad  \text{for } 1<j<N.
\end{equation}

\emph{Case $j=1$:} This resembles the analysis above of $\tau_{j,0}$; one starts from~\eqref{tauj0} with~$j=1$ there. The only change is that now one invokes a standard bound on Euler's Beta function \citep[Theorem D.6]{Diet10} to see that
\begin{align*}
\left|\frac1{\Gamma(2-\delta)}\int_{x_0}^{x_{1}} (x_1-s)^{1-\delta} u''(s)\,ds \right|
    &\le \frac{C}{\Gamma(2-\delta)}\int_{x_0}^{x_{1}} (x_1-s)^{1-\delta} s^{\delta-2} \,ds \\
    &=\frac{C\,\Gamma(2-\delta)\Gamma(\delta-1)}{\Gamma(2-\delta)}\le C,
\end{align*}
while as before
$$
\left| \frac{d_1}{\hdg} \left[ u(x_{2}) - 2u(x_{1}) + u(x_{0})\right] \right| \le C d_1 = C.
$$
 By the mean value theorem, for some $\eta_{7}\in (x_0,x_2)$ one has
\begin{equation}\label{R1}
|R_1| = \left|b_1D^\aleph u_1 - (bu')(x_1)\right|
     \le \Vert b \Vert_\infty \, \vert u'(\eta_{7})-u'(x_1)\vert
     \le C
\end{equation}
for some constant $C$, because $u \in C^1[0,1]$. Combining the above bounds, we obtain
\begin{equation}\label{tau1}
\vert \tau_1 \vert \le \vert \tau_{1,0}\vert +\vert R_1\vert\le C.
\end{equation}

\end{proof}

\begin{remark}\label{rem:ShenLiu}
In \citet[Lemma~3]{SL0405} Taylor expansions are used to prove that $|\tau_j| \le Ch$ for $j=1,2,\dots, N-1$ under the implicit assumptions that $u'''$ and $u^{(4)}$ are bounded on~$[0,1]$. An inspection of the argument shows that it can be modified slightly to yield the same result under the assumption that $u'''$ is bounded on $[0,1]$; no assumption on $u^{(4)}$ is needed. Nevertheless the assumption that $u'''$ is bounded on $[0,1]$ is very strong and restricts the applicability of this result to special cases of~\eqref{prob} whose solutions are exceptionally smooth.
\end{remark}

We can now prove that our finite difference method is $O\left(h^{\delta-1}\right)$ accurate in the discrete maximum norm.
\begin{theorem}\label{thm:Error}
Let $b,c, f \in C^{q,\mu}(0,1]$ for some integer $q \ge 2$ and $\mu \le 2-\delta$. Assume that $c \ge 0,\ \alpha_1 \ge 0$ and the condition~\eqref{AlRcondition} is satisfied. Then the error in the discretization $A\vec u = \vec f$ of~\eqref{prob} satisfies
$$
\Vert u -\vec u \Vert_{\infty,d} \le C h^{\delta-1},
$$
where $\|u -\vec u\|_{\infty,d} := \max_{j=0,1,\dots,N}|u(x_j)-u_j|$.
\end{theorem}

\begin{proof}
Recalling \eqref{djk} and \eqref{Hessenberga}, the construction of Section~\ref{sec:monotoneA} added the multiple
$$
\frac{-a_{j0}}{1+\alpha_0 h^{-1}}
    = \frac1{1+\alpha_0 h^{-1}} \left[ \frac{d_j}{\hdg} +\epsilon_{j1}\frac{b_1+|b_1|}{2h} \right]
$$
of row $0$ of $A$ to row $j$ for $j=1,2,\dots, N-1$, yielding an M-matrix $A'$. When this construction is applied to the system of equations~\eqref{tau}, one modifies $\tau_j$ to
\begin{equation}\label{tau2}
\tau_j' := \tau_j + \frac{\tau_0}{1+\alpha_0 h^{-1}} \left[ \frac{d_j}{\hdg} +\epsilon_{j1}\frac{b_1+|b_1|}{2h} \right]  \quad\text{for } j=1,2,\dots, N-1,
\end{equation}
and then
\begin{equation}\label{tau3}
A'(u-\vec u) = \vec\tau', \ \text{\ with } \vec\tau' := (\tau_0\ \tau_1'\ \tau_2' \dots \tau_{N-1}'\  \tau_N)^T.
\end{equation}

For $j=1,2,\dots, N-1$, the proof of Theorem~\ref{th:Amonotone} shows that the $j^\text{th}$ row sum of~$A'$ is
\begin{equation}\label{jrowsum}
\left( \frac{\alpha_0 h^{-1}}{1+\alpha_0 h^{-1}} -1 \right)a_{j0}+c_j
    = \frac{-a_{j0}}{1+\alpha_0 h^{-1}}+c_j
    = \frac1{1+\alpha_0 h^{-1}} \left[ \frac{d_j}{\hdg} +\epsilon_{j1}\frac{b_1+|b_1|}{2h} \right]+c_j\,.
\end{equation}
Thus the value of the $j^\text{th}$  row sum depends strongly on $j$. We rescale rows $1,2,\dots, N-1$ in the system of equations~\eqref{tau3} by multiplying the $j^\text{th}$ equation by
$$
\frac{(1+\alpha_0 h^{-1})\hdg}{d_j}\,,
$$
so that by~\eqref{jrowsum}  each of these row sums is now $O(1)$ and equals
\begin{equation}\label{jrowsumtilde}
1+ \frac{\hdg}{d_j}\left[ \epsilon_{j1}\frac{b_1+|b_1|}{2h} + (1+\alpha_0 h^{-1})c_j \right]\ge 1.
\end{equation}
Write $\tilde A$ for the rescaled matrix of the system of equations and $\vec{\tilde \tau}$ for the rescaled right-hand side, so now we have
\begin{equation}\label{tau4}
\tilde A(u-\vec u) = \vec{\tilde\tau}, \ \text{\ with } \vec{\tilde\tau} := (\tau_0\ {\tilde\tau}_1\ \tilde{\tau}_2 \dots \tilde{\tau}_{N-1} \ \tau_N)^T
\end{equation}
and for $j=1,2,\dots, N-1$,
$$
\tilde{\tau}_j = \frac{(1+\alpha_0 h^{-1})\hdg}{d_j}\, \tau_j'
    = \frac{(1+\alpha_0 h^{-1})\hdg}{d_j}\,\tau_j + \left[1+ \epsilon_{j1}\frac{b_1+|b_1|}{2h}\cdot \frac{\hdg}{d_j} \right]\tau_0
$$
by~\eqref{tau2}. Hence, Lemma~\ref{lem:TruncError}  and  $d_j  \ge (2-\delta)j^{1-\delta}\ge (2-\delta)2^{1-\delta}(j-1)^{1-\delta}$ for $j\ge 2$ imply that
\begin{equation}\label{tildetau}
\vert \tilde{\tau}_j \vert \le C h^{\delta-1} \quad{\text{for }} j=1,2,\dots, N-1.
\end{equation}

But the off-diagonal entries of $\tilde A$ are non-positive and $\tilde A(1\ 1\dots 1)^T \ge (1\ 1\dots 1)^T > \vec 0$ by~\eqref{jrowsumtilde}, so  $\tilde A$ is an M-matrix; furthermore, it follows that in the standard matrix norm notation $\|\cdot\|_\infty$ one has $\|(\tilde A)^{-1}\|_\infty \le 1$ ---see, e.g., \citet[Lemma 2.1]{AK90}. Consequently~\eqref{tau4} and the bound~\eqref{tildetau}, together with $|\tau_0| \le Ch^{\delta-1}$ and $|\tau_N| \le C \alpha_1 h$ from Lemma~\ref{lem:TruncError},  imply that $\Vert u -\vec u \Vert_{\infty,d} \le C h^{\delta-1}$, as desired. \end{proof}

\begin{remark}\label{rem:comments_conditionN}
The order of convergence proved in  Theorem \ref{thm:Error} is the same order (in $\|\cdot\|_{\infty,d}$) as that proved in the norm $\|\cdot\|_\infty$ for the simplest collocation method ($m=1$) in \citet[Theorem 4.1]{PT12}
on a uniform mesh, but Theorem~\ref{thm:Error} places no condition on the mesh diameter, while the results of \cite{PT12} implicitly require $h^{\delta-1}$ to be smaller than some fixed constant---this may be restrictive when~$\delta$ is near~1.  This mesh condition arises because the proofs of the main convergence results of \cite{PT12} rely on the property that, for sufficiently large $N$, the operator $I-{\cal P}_N T$ is invertible and $\|(I-{\cal P}_N T)^{-1}\|$ is bounded in a certain  operator norm; to verify this property, the authors appeal to a standard argument from \cite[Lemma 3.2]{BPV01}, but on a uniform mesh this relies on inequality~(3.12) of \cite{BPV01} with $r=1,\ \nu =2-\delta$ and $m \ge 1$, and consequently ``for sufficiently large $N$" is equivalent to ``for $h^{\delta-1}$ sufficiently small". Note however that the mesh restriction is less demanding for the graded meshes that are also considered in \cite{PT12}.
\end{remark}

\section{Numerical results}\label{sec:numerical}
We first consider a problem whose solution $u$ lies in $C^1[0,1]\cap C^\infty(0,1]$, but $u\notin C^2[0,1]$ and the behaviour of $u$ mimics exactly the behaviour of the estimates of the solution in Corollary~\ref{cor:PTintegerbounds}; cf.~Example~\ref{exa:interpolationsharp}.

\emph{Test Problem 1.}
\begin{subequations}\label{test}
\begin{align}
-D_\ast^{\delta} u(x)+x^2u'(x)+(1+x)u(x) &= f(x) \ \text{ on }(0,1),  \\
u(0)-\ds\frac{1}{\delta-1} u'(0) &= \gamma_0 \ \text{ and } \ u(1)+u'(1) = \gamma_1,
\end{align}
\end{subequations}
where the function $f$ and the constants $\gamma_0$ and $\gamma_1$ are chosen such that the exact solution of~\eqref{test} is $u(x)=x^{\delta}+x^{2\delta-1}+1+3x-7x^2+4x^3+x^4$.


\begin{figure}[h]
\centering
\includegraphics[width=0.5\textwidth]{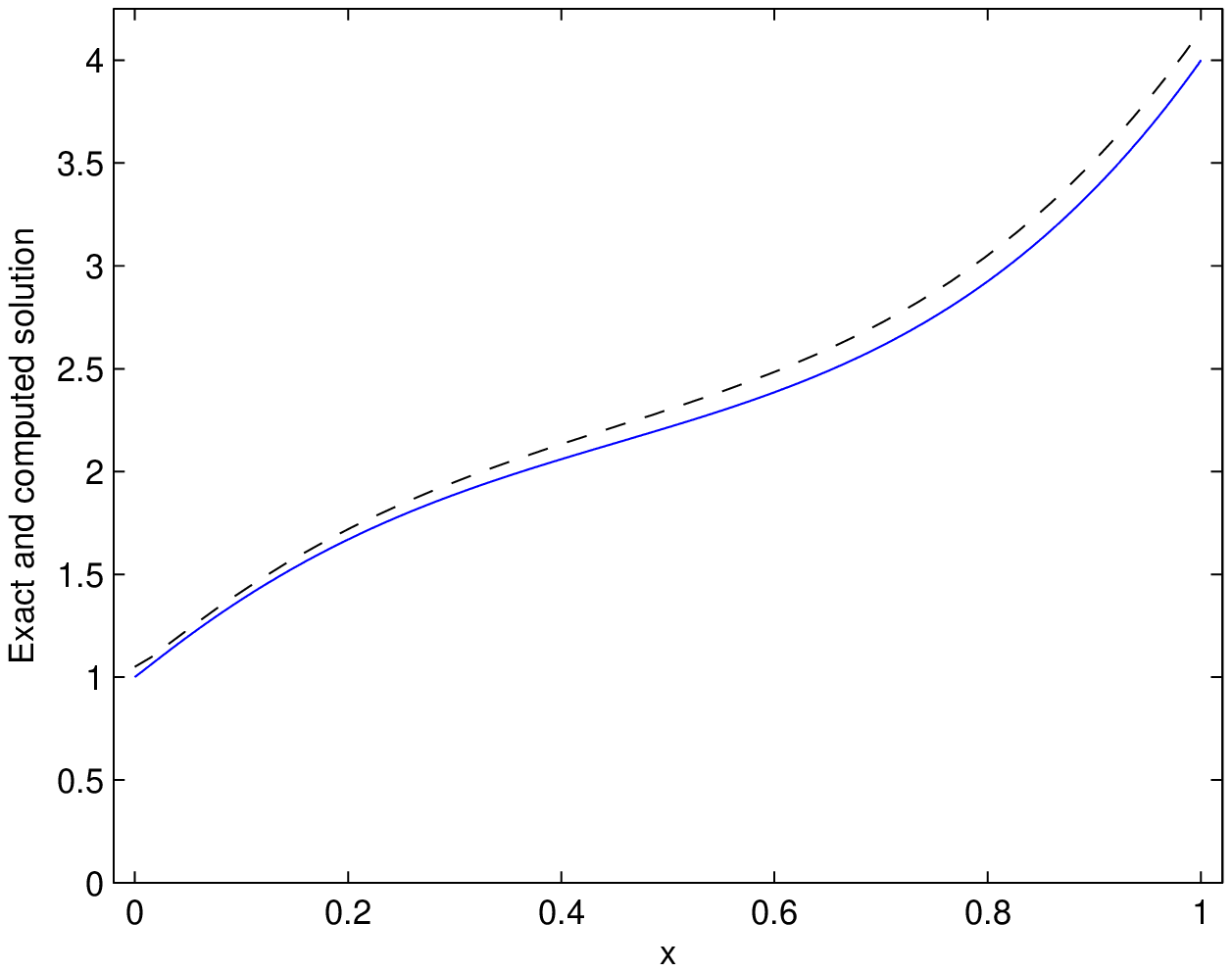}%
\includegraphics[width=0.5\textwidth]{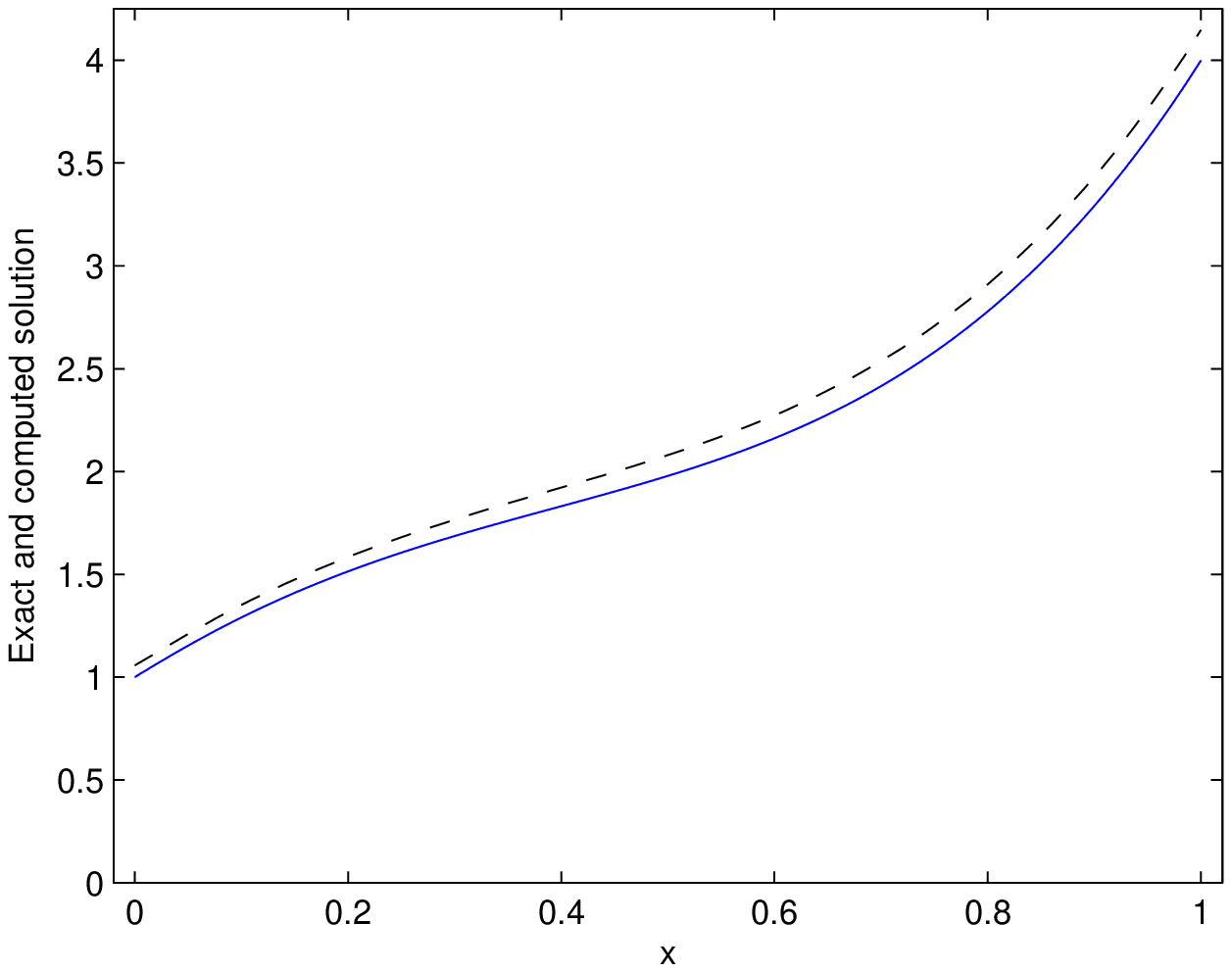}%
\caption{Exact (solid) and computed (dashed) solutions of (\ref{test}) for $N=64$ with $\delta=1.1$ (left figure) and $\delta=1.4$ (right figure)}
\label{fig1}
\end{figure}

\begin{figure}[h]
\centering
\includegraphics[width=0.5\textwidth]{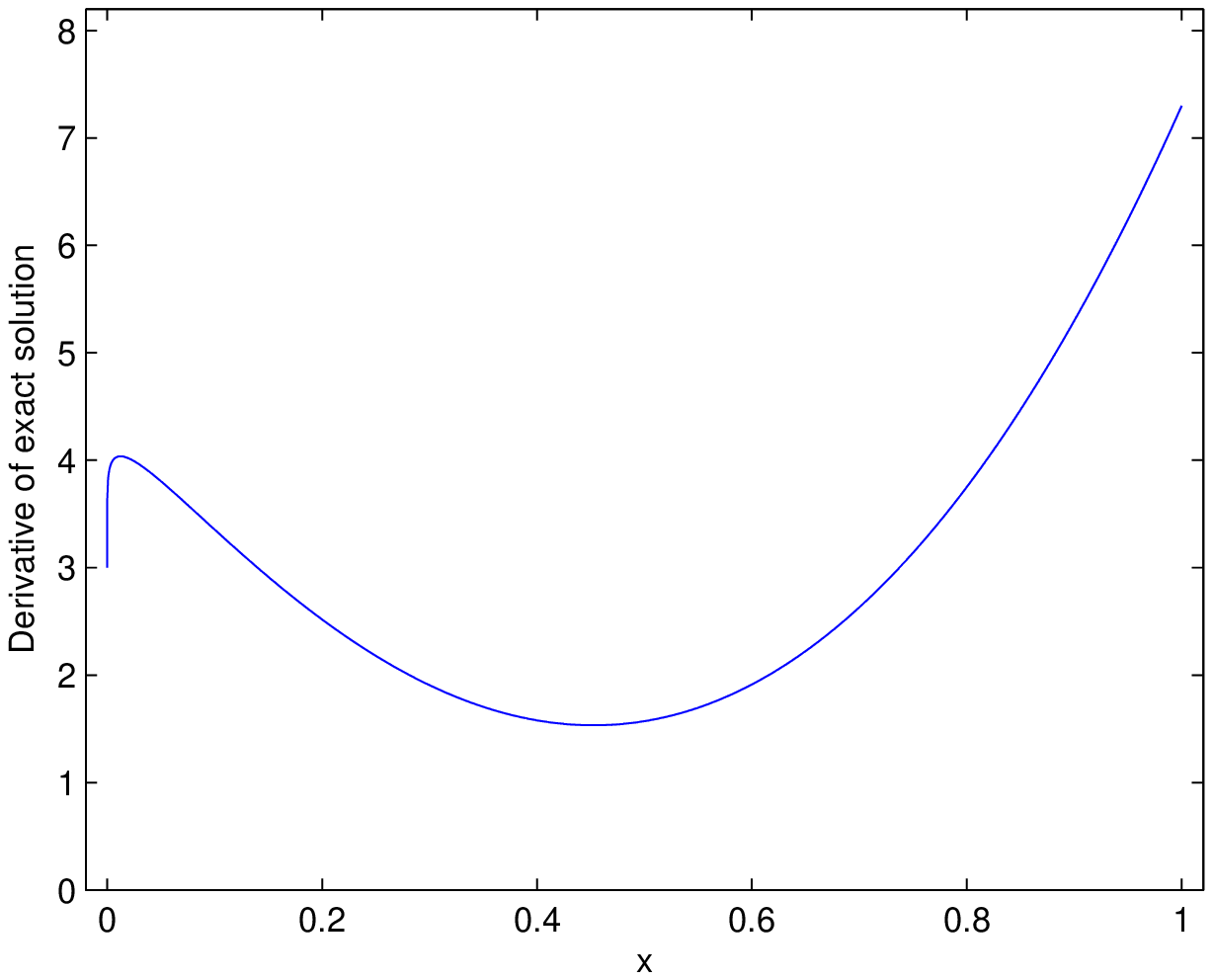}%
\includegraphics[width=0.5\textwidth]{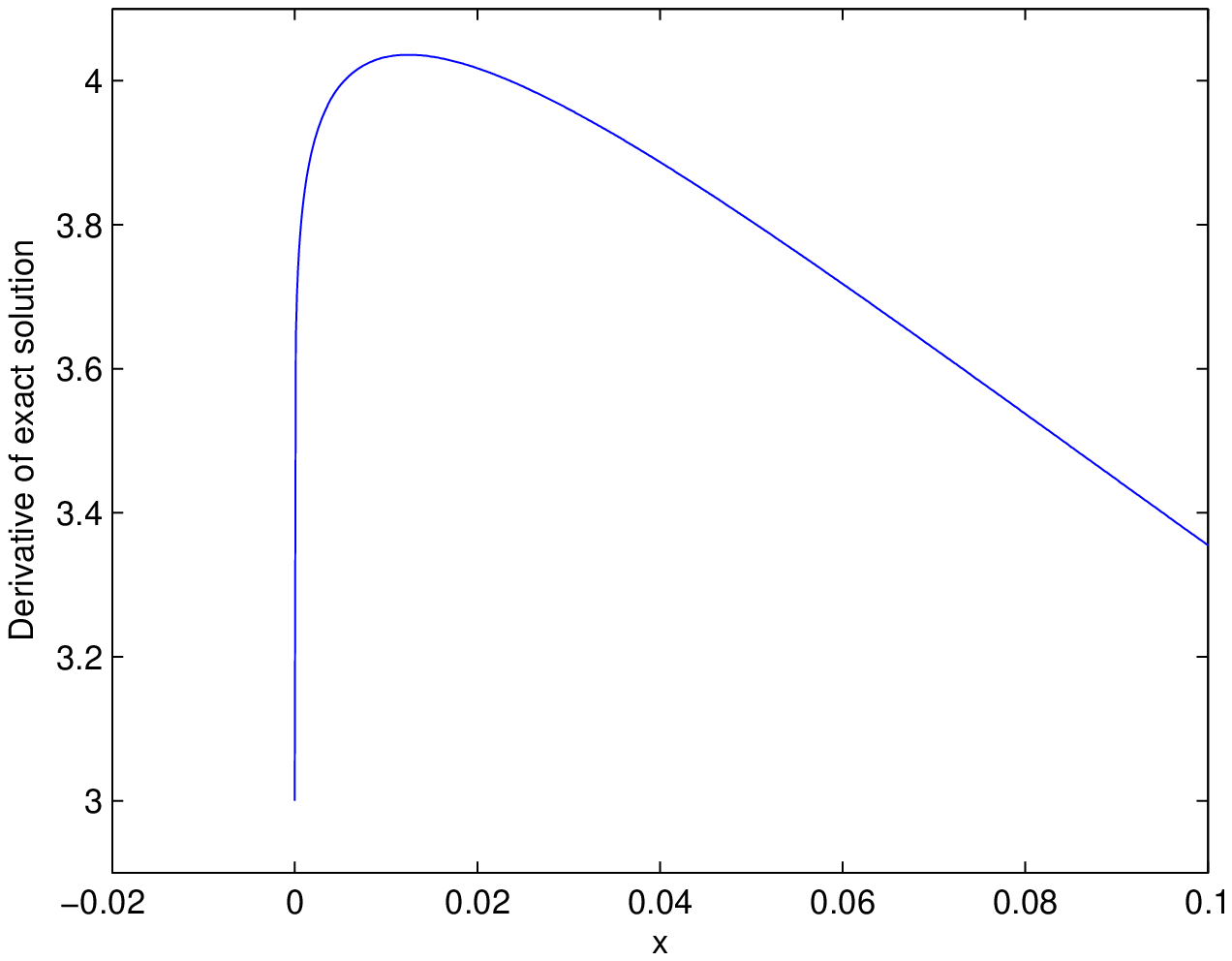}%
\caption{Derivative $u'(x)$ of the solution of (\ref{test}) with $\delta=1.1$, for $x \in [0,1]$ (left figure) and $x$ near 0 (right figure)}
\label{fig2b}
\end{figure}

\begin{figure}[h]
\centering
\includegraphics[width=0.5\textwidth]{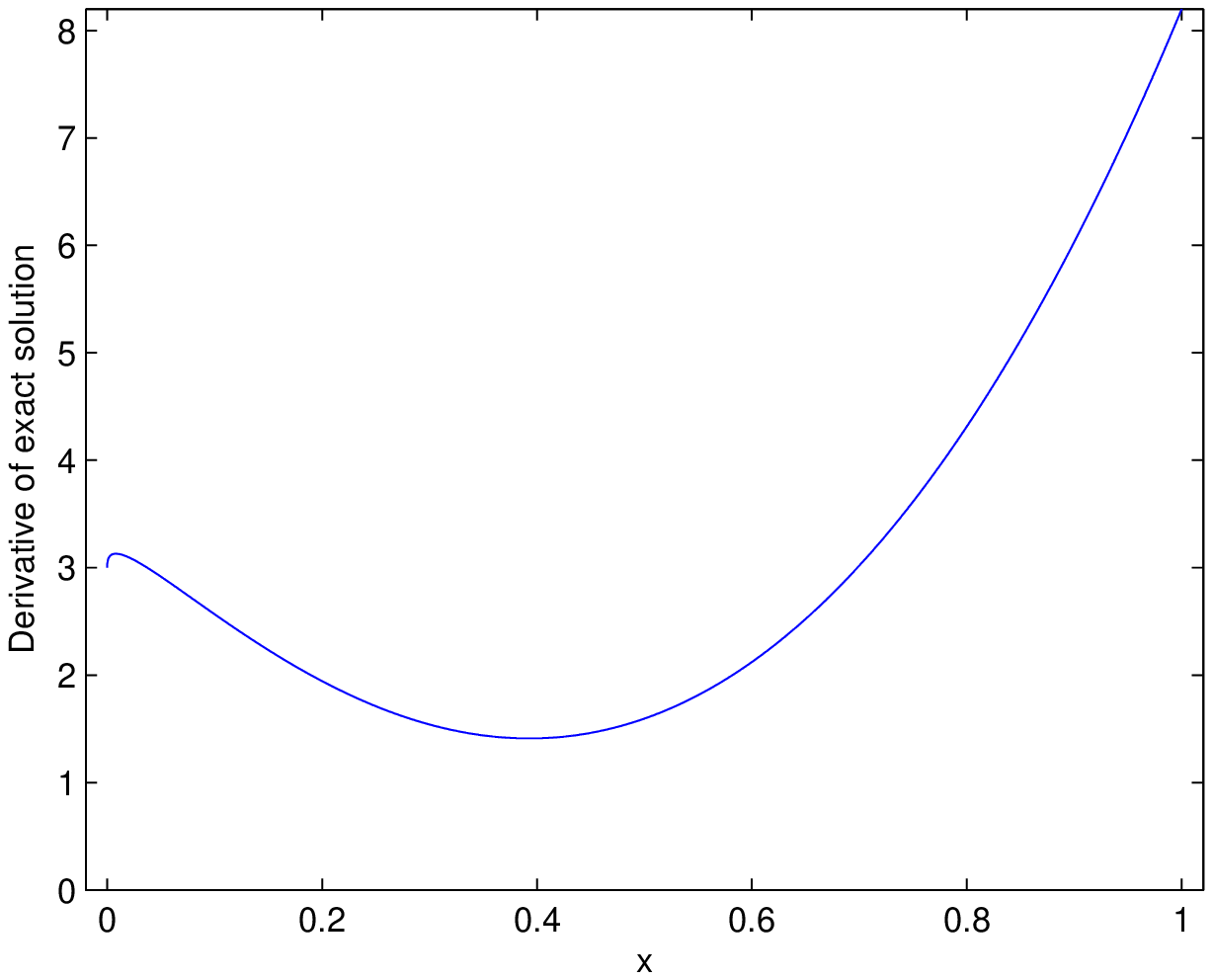}%
\includegraphics[width=0.5\textwidth]{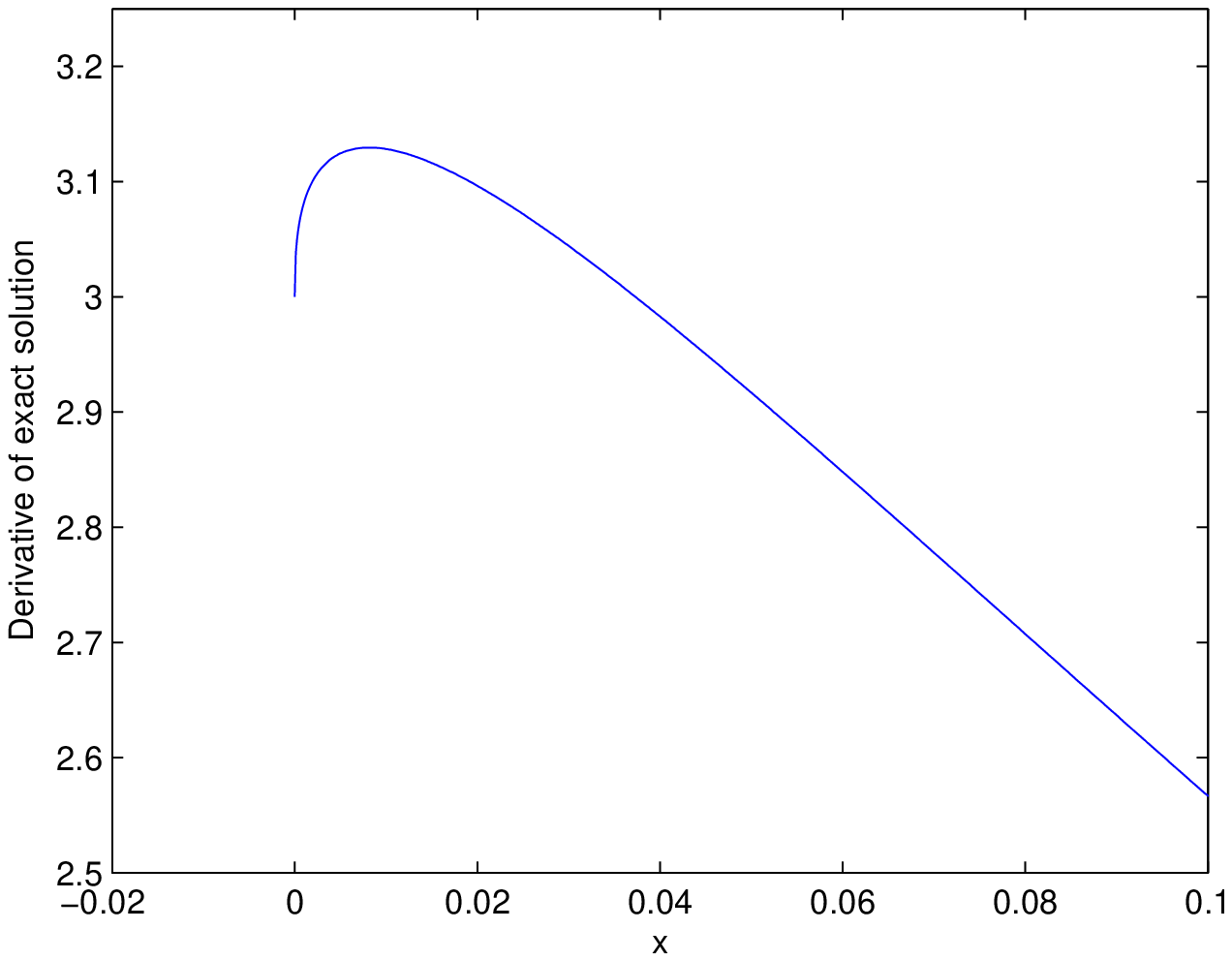}%
\caption{Derivative $u'(x)$ of the solution of (\ref{test}) with $\delta=1.4$, for $x \in [0,1]$ (left figure) and $x$ near 0 (right figure)}
\label{fig2}
\end{figure}

The numerical solution $\{u_j\}_{j=0}^N$ of problem (\ref{test}) is computed on a uniform mesh of width $h=1/N$ as described in Section~\ref{sec:propertiesA}.
Figure~\ref{fig1} shows the exact solution $u$ for~$\delta=1.1$ (left figure) and~$\delta=1.4$ (right figure), together with the respective solutions computed by our finite difference method for $N=64$. Figures~\ref{fig2b} (left) and \ref{fig2} (left) show the derivative $u'$ of the exact solution for $\delta=1.1$ and $\delta=1.4$, respectively. We observe that the solution $u$ and its derivative $u'$ are bounded functions. In Figures~\ref{fig2b} (right) and \ref{fig2} (right) a zoom of $u'(x)$ in the vicinity of $x=0$ is displayed, and we can observe  a vertical tangent at this point in both cases.

Table \ref{tb1} presents numerical results for several values of the parameters~$\delta$ and~$N$.
Each table entry shows first the maximum pointwise error
$$
e_N^\delta := \max_{0\le j\le N} \vert u(x_j) - u_j\vert,
$$
and then the order of convergence, which is computed in the standard way:
$$
p_N^\delta = \log _2 \left(\frac{e_N^\delta}{e_{2N}^\delta}\right).
$$
To show that the numerical results do not depend strongly on the value of $\delta$, we have also computed the uniform errors for the set of values of the parameter $\delta$ considered in Table \ref{tb1} and the corresponding orders of convergence. These values are defined for each $N$  by
$$
e_N=\max_\delta\, e_N^\delta \quad \hbox{and} \quad p_N=\log _2 \left(\frac{e_{N}}{e_{2N}}\right)
$$
and they appear in the last row of Table \ref{tb1}.

These numerical results show that one obtains first-order convergence for each value of $\delta$ considered in Table~\ref{tb1}, and this convergence is  uniform in $\delta$.
Figure~\ref{fig3} exhibits the pointwise errors in the computed solution for $\delta=1.1, 1.4$ and $N=64, 128$, to show how the error varies within $[0,1]$.

The first-order convergence of Table~\ref{tb1} is much better than the $O(h^{\delta-1})$ convergence guaranteed by Theorem~\ref{thm:Error}, but our second numerical example will show that the rate of convergence can indeed deteriorate when $\delta$ is close to $1$.

\begin{table}[h]
\caption{Test Problem (\ref{test}): Maximum and uniform errors $e_N^\delta, \, e_N$  and their orders of convergence $p_N^\delta, \, p_N$}
\begin{center}{\scriptsize \label{tb1}
\begin{tabular}{|c||c|c|c|c|c|c|}
 \hline    & N=64 & N=128 & N=256 & N=512 & N=1024 & N=2048 \\
\hline \hline $\delta=1.1$
&1.464E-001 &7.547E-002 &3.843E-002 &1.941E-002 &9.761E-003 &4.895E-003 \\
&0.956&0.974&0.985&0.992&0.996&0.998
\\ \hline $\delta=1.2$
&1.455E-001 &7.443E-002 &3.769E-002 &1.897E-002 &9.522E-003 &4.770E-003 \\
&0.967&0.982&0.990&0.995&0.997&0.999
\\ \hline $\delta=1.3$
&1.457E-001 &7.427E-002 &3.754E-002 &1.888E-002 &9.472E-003 &4.744E-003 \\
&0.972&0.984&0.991&0.995&0.997&0.999
\\ \hline $\delta=1.4$
&1.466E-001 &7.469E-002 &3.776E-002 &1.900E-002 &9.536E-003 &4.779E-003 \\
&0.973&0.984&0.991&0.995&0.997&0.998
\\ \hline $\delta=1.5$
&1.476E-001 &7.534E-002 &3.816E-002 &1.924E-002 &9.669E-003 &4.852E-003 \\
&0.971&0.981&0.988&0.992&0.995&0.997
\\ \hline $\delta=1.6$
&1.479E-001 &7.573E-002 &3.849E-002 &1.947E-002 &9.816E-003 &4.937E-003 \\
&0.965&0.976&0.983&0.988&0.991&0.994
\\ \hline $\delta=1.7$
&1.459E-001 &7.508E-002 &3.836E-002 &1.950E-002 &9.876E-003 &4.988E-003 \\
&0.958&0.969&0.976&0.981&0.985&0.988
\\ \hline $\delta=1.8$
&1.392E-001 &7.197E-002 &3.698E-002 &1.891E-002 &9.637E-003 &4.896E-003 \\
&0.951&0.961&0.967&0.973&0.977&0.980
\\ \hline $\delta=1.9$
&1.236E-001 &6.389E-002 &3.288E-002 &1.686E-002 &8.628E-003 &4.405E-003 \\
&0.952&0.958&0.963&0.967&0.970&0.973
\\ \hline $e_N$
&1.479E-001 &7.573E-002 &3.849E-002 &1.950E-002 &9.876E-003 &4.988E-003 \\
$p_N$ &0.965&0.976&0.981&0.981&0.985&0.988\\ \hline \hline
\end{tabular}}
\end{center}
\end{table}

\begin{figure}[hta!]
\centering
\includegraphics[width=0.5\textwidth]{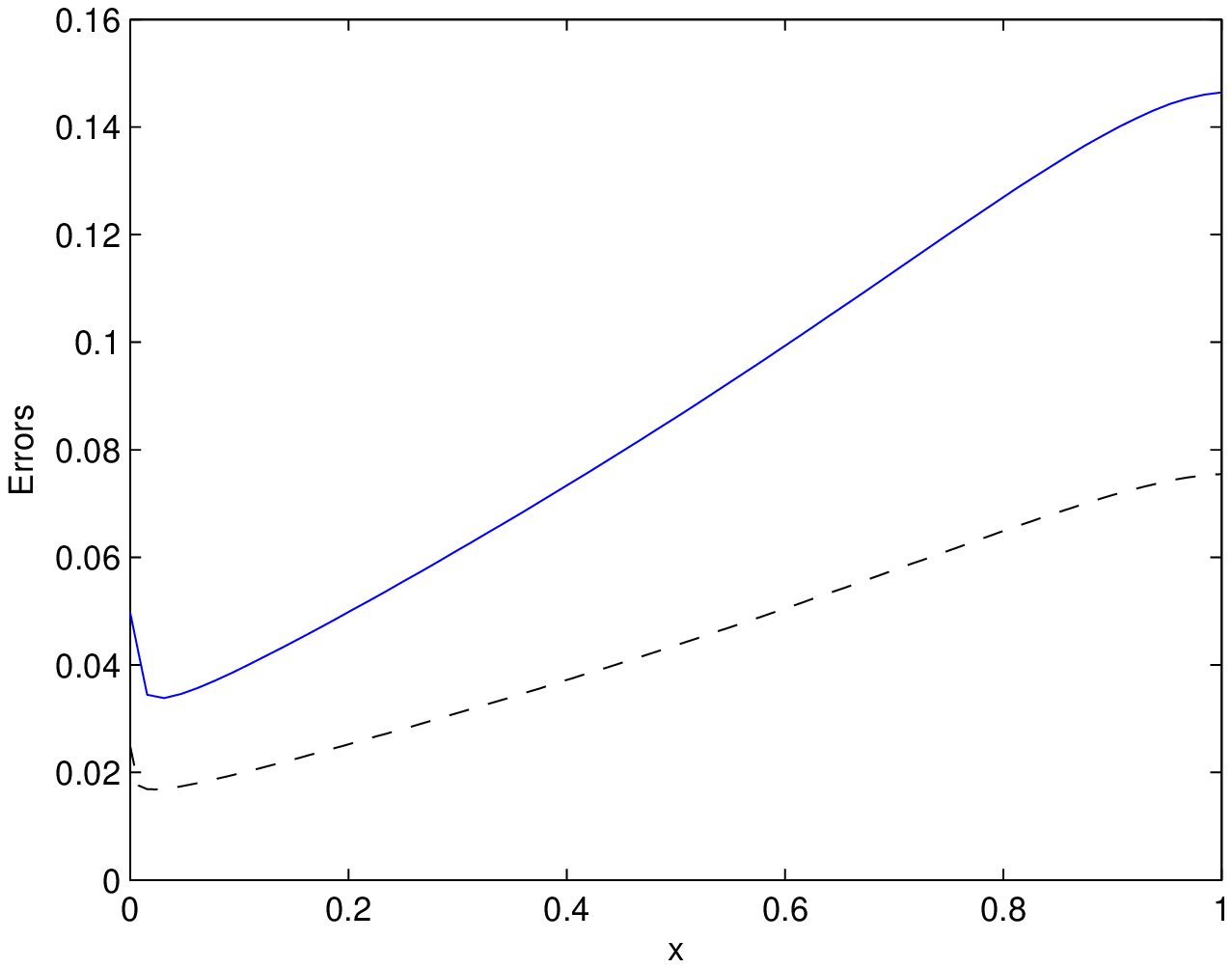}%
\includegraphics[width=0.5\textwidth]{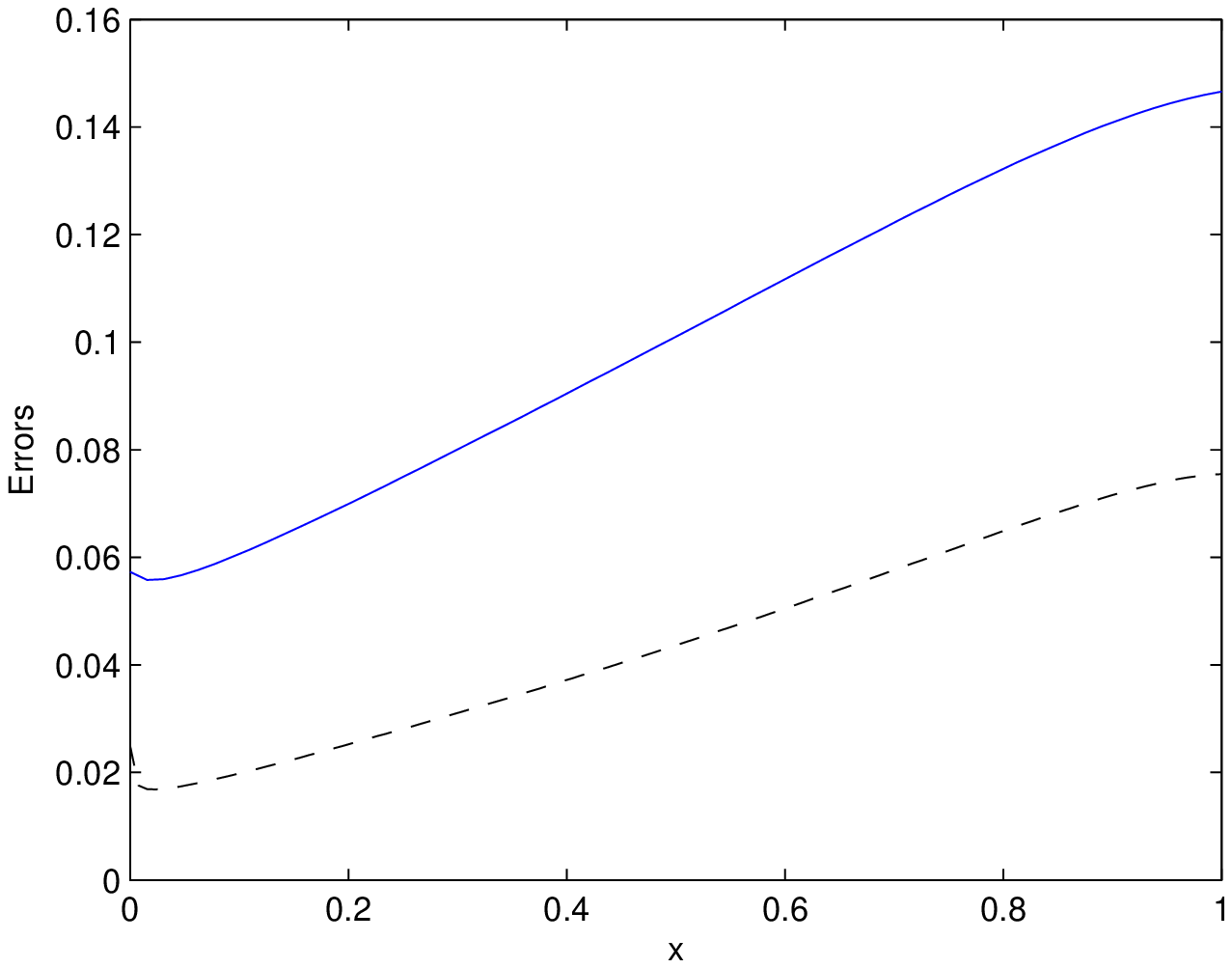}%
\caption{Pointwise errors in computed solutions for test problem (\ref{test}) with $\delta=1.1$ (left figure) and $\delta=1.4$ (right figure). The values of the discretization parameters are $N=64$ (solid  line) and $N=128$ (dashed line).}
\label{fig3}
\end{figure}

\newpage
\emph{Test Problem 2.}

Consider the constant-coefficient problem
\begin{subequations}\label{test2}
\begin{align}
-D_\ast^{\delta} u(x)+ 2 u'(x)+3u(x) &= 1.25 \ \text{ on }(0,1),  \\
u(0)-\ds\frac{1}{\delta-1} u'(0) &= 0.4 \ \text{ and } \ u(1)= 1.7.
\end{align}
\end{subequations}
As the exact solution of \eqref{test2} is unknown, in order to estimate the errors  of the solution~$\{u_j\}^N_{j=0}$ computed on a uniform mesh of width $h=1/N$ by our finite difference method, we use the two-mesh principle \cite[Section 5.6]{FHMORS00}: on a uniform mesh of width $h/2$, compute the numerical solution~$\{z_j\}^{2N}_{j=0}$ with the same method and hence the two-mesh differences
$$
d_N^\delta := \max_{0\le j\le N} \vert u_j - z_{2j}\vert;
$$
from these values we estimate the orders of convergence by
$$
q_N^\delta = \log _2 \left(\frac{d_N^\delta}{d_{2N}^\delta}\right).
$$
The uniform two-mesh differences and their corresponding uniform orders of convergence are computed analogously to Table~\ref{tb1} and denoted by~$d_N$ and~$q_N$, respectively. The numerical results obtained are displayed in Table~\ref{tb2} and we observe that the finite difference method is again convergent but a significant decrease  in the order of convergence occurs for values of~$\delta$ close to~1 and practical values of~$N$.

\begin{table}[h]
\caption{Test Problem (\ref{test2}): Maximum and uniform two-mesh differences $d_N^\delta, \, d_N$  and their orders of convergence $q_N^\delta, \, q_N$}
\begin{center}{\scriptsize \label{tb2}
\begin{tabular}{|c||c|c|c|c|c|c|}
 \hline    & N=64 & N=128 & N=256 & N=512 & N=1024 & N=2048 \\
\hline \hline $\delta=1.1$
&2.304E-001 &2.277E-001 &2.200E-001 &2.075E-001 &1.913E-001 &1.702E-001 \\
&0.017&0.050&0.084&0.117&0.168&0.255
\\ \hline $\delta=1.2$
&1.289E-001 &1.032E-001 &7.512E-002 &4.919E-002 &2.935E-002 &1.631E-002 \\
&0.321&0.458&0.611&0.745&0.847&0.915
\\ \hline $\delta=1.3$
&6.557E-002 &4.240E-002 &2.509E-002 &1.387E-002 &7.337E-003 &3.782E-003 \\
&0.629&0.757&0.855&0.919&0.956&0.977
\\ \hline $\delta=1.4$
&3.635E-002 &2.125E-002 &1.167E-002 &6.158E-003 &3.174E-003 &1.614E-003 \\
&0.775&0.864&0.922&0.956&0.976&0.986
\\ \hline $\delta=1.5$
&2.271E-002 &1.265E-002 &6.749E-003 &3.509E-003 &1.795E-003 &9.104E-004 \\
&0.844&0.906&0.944&0.967&0.980&0.988
\\ \hline $\delta=1.6$
&1.548E-002 &8.418E-003 &4.443E-003 &2.300E-003 &1.177E-003 &5.976E-004 \\
&0.879&0.922&0.950&0.967&0.978&0.985
\\ \hline $\delta=1.7$
&1.110E-002 &5.968E-003 &3.140E-003 &1.628E-003 &8.356E-004 &4.262E-004 \\
&0.895&0.927&0.948&0.962&0.971&0.978
\\ \hline $\delta=1.8$
&8.063E-003 &4.305E-003 &2.264E-003 &1.178E-003 &6.078E-004 &3.120E-004 \\
&0.905&0.927&0.943&0.954&0.962&0.969
\\ \hline $\delta=1.9$
&5.650E-003 &2.978E-003 &1.557E-003 &8.090E-004 &4.184E-004 &2.156E-004 \\
&0.924&0.936&0.945&0.951&0.957&0.961
\\ \hline $d_N$
&2.304E-001 &2.277E-001 &2.200E-001 &2.075E-001 &1.913E-001 &1.702E-001 \\
$q_N$ &0.017&0.050&0.084&0.117&0.168&0.255\\ \hline \hline
\end{tabular}}
\end{center}
\end{table}

\section{Conclusions}\label{sec:conclusions}
In this paper we discussed a two-point boundary value problem whose highest-order derivative was a Caputo fractional derivative of order $\delta \in (1,2)$. A comparison principle was proved for this differential operator on its domain $[0,1]$ provided that the boundary conditions satisfied certain restrictions. Then we derived sharp a priori bounds on the integer-order derivatives of the solution $u$ of the boundary value problem, using elementary analytical techniques to extract this information from previously-known bounds on the integer-order derivatives of $D^\delta_\ast u$. These new bounds were used to analyse a finite difference scheme for this problem via a truncation error analysis, but this analysis was complicated by the awkward fact that $u''(x)$ and $u'''(x)$  blow up at the boundary $x=0$ of the domain $[0,1]$. We were able to prove that our finite difference method was $O(h^{\delta-1})$ accurate at the nodes of our mesh ($h$ is the mesh width and the mesh is uniform), but our numerical experience has been that the method is  often more accurate; we have observed first-order convergence for all values of $\delta$ in several numerical examples, though one can have some deterioration in the rate of convergence when $\delta$ is near~1, as we saw in Test Problem 2.

In future work in~\cite{future_paper_GS} and other papers we shall discuss the use of alternative difference approximations of the convective term of~\eqref{proba}, investigate why the rate of convergence of~\eqref{fulldisc} is sometimes first order for all values of $\delta$, and  extend our approach to higher-order difference schemes and to graded meshes (cf.~\citet{PT12}).

\bibliographystyle{plainnat}
\bibliography{Caputo2ptbvp}

\end{document}